\documentclass[11pt,letterpaper]{amsart}

\usepackage[T1]{fontenc}
\usepackage{lmodern}
\usepackage{microtype}
\usepackage{mathtools}
\usepackage{amssymb}
\usepackage{mathrsfs}
\usepackage{enumitem}
\usepackage{url}
\usepackage[colorlinks=true,linkcolor=blue,citecolor=blue,urlcolor=blue]{hyperref}
\usepackage{aliascnt}

\numberwithin{equation}{section}

\newtheorem{theorem}{Theorem}[section]
\newaliascnt{proposition}{theorem}
\newtheorem{proposition}[proposition]{Proposition}
\aliascntresetthe{proposition}
\newaliascnt{lemma}{theorem}
\newtheorem{lemma}[lemma]{Lemma}
\aliascntresetthe{lemma}
\newaliascnt{corollary}{theorem}
\newtheorem{corollary}[corollary]{Corollary}
\aliascntresetthe{corollary}
\theoremstyle{definition}
\newaliascnt{definition}{theorem}
\newtheorem{definition}[definition]{Definition}
\aliascntresetthe{definition}
\theoremstyle{remark}
\newaliascnt{remark}{theorem}

\aliascntresetthe{remark}

\newcommand{\Ric}{\operatorname{Ric}}
\newcommand{\Rm}{\operatorname{Rm}}
\newcommand{\inj}{\operatorname{inj}}
\newcommand{\Hess}{\operatorname{Hess}}

\newcommand{\Vol}{\operatorname{Vol}}
\newcommand{\Int}{\operatorname{Int}}
\newcommand{\genus}{\operatorname{genus}}
\newcommand{\spt}{\operatorname{spt}}

\newcommand{\tr}{\operatorname{tr}}

\newcommand{\eps}{\varepsilon}

\newcommand{\Q}{\mathbb Q}
\newcommand{\R}{\mathbb R}
\newcommand{\Z}{\mathbb Z}
\newcommand{\dd}{\,d}

\title[Contractible 3-Manifolds and Handlebody Interiors]{Topological Rigidity of Contractible 3-Manifolds and Handlebody Interiors under Nonnegative Scalar Curvature}
\author{Jiangcheng You and Heng Zhang}
\date{\today}

\address{School of Mathematical Sciences, University of Science and Technology of China, Hefei, China}
\email{yjcmp@mail.ustc.edu.cn}

\address{School of Mathematical Sciences, University of Science and Technology of China, Hefei, China}
\email{hengz@mail.ustc.edu.cn}


\hypersetup{
	pdftitle={Topological Rigidity of Contractible 3-Manifolds and Handlebody Interiors under Nonnegative Scalar Curvature},
	pdfauthor={Jiangcheng You and Heng Zhang}
}

\begin{document}
	
	\begin{abstract}
		We prove that a contractible $3$-manifold admitting a complete Riemannian metric of nonnegative scalar curvature is diffeomorphic to $\R^3$. We also prove that, if the interior of a compact handlebody of genus $\gamma$ admits such a metric, then $\gamma\leq1$. This answers two open questions posed by Wang and Gromov, respectively.
	\end{abstract}
	
	\maketitle
	\tableofcontents
	
	\section{Introduction}\label{sec:introduction}
	
	The classification of complete open $3$-manifolds admitting metrics of nonnegative scalar curvature is a basic problem in scalar curvature geometry. A fundamental source of the subject is the minimal surface method of Schoen--Yau. In dimension three, the stability inequality for minimal surfaces, together with the Gauss equation and the Gauss--Bonnet theorem, converts scalar curvature information into strong topological restrictions on the surface and on the ambient manifold \cite{SchoenYau,SchoenYauStructure,SchoenYauComplete3D}. Their Wisconsin survey places these results in the broader program of determining the topology of complete manifolds with positive scalar curvature \cite{SchoenYauWisconsin}. Another foundational work is that of Gromov–Lawson, who obtained complementary obstructions via index theoretic methods \cite{GromovLawson}. Nevertheless, a general description of the possible complete open $3$-manifolds is still not known; see Yau's problem list \cite{YauProblems} and the recent discussions in \cite{CLX,YanZhu}. The following two natural test cases have played a central role.

	\medskip
	\noindent\textbf{Wang's question \cite{WangI,WangII}.}
	If a contractible $3$-manifold admits a complete Riemannian metric with nonnegative scalar curvature, must it be diffeomorphic to $\R^3$?
	
	The contractible case is subtle because the usual fundamental group obstructions disappear, while the topology at infinity may still be highly nontrivial. The uniformly positive scalar curvature case is covered by the work of Gromov--Lawson and Chang--Weinberger--Yu \cite{GromovLawson,ChangWeinbergerYu}. In that regime, Bessi\`eres--Besson--Maillot obtained a classification under a bounded-geometry assumption \cite{BessieresBessonMaillot}, and Wang later established the general classification without that auxiliary hypothesis \cite{WangUniformPSC}. Beyond the uniformly positive setting, Wang developed the contractible problem in depth and proved the desired rigidity for important classes, including contractible $3$-manifolds governed by solid-torus exhaustions \cite{WangI,WangII}; see also \cite{CLX}.
	
	\medskip
	\noindent\textbf{Gromov's question \cite[\S3.10.2]{GromovLectures}.}
	Let $M_\gamma$ be the interior of a compact handlebody of genus $\gamma$. If $M_\gamma$ admits a complete Riemannian metric with nonnegative scalar curvature, must one have $\gamma\leq1$?
	
	\medskip
	Chodosh--Lai--Xu \cite{CLX} answered both questions affirmatively under the bounded-geometry assumption
	\begin{equation}\label{eq:intro-BG}
		|\Rm_g|\leq\Lambda,
		\qquad
		\inj_g\geq\Lambda^{-1}.
	\end{equation}
	Their proof combines the classical minimal surface obstructions with the maximal weak inverse mean curvature flow. By using Geroch monotonicity, compactness, and metric perturbations in the nonproper alternatives, they construct exhaustions whose relevant boundary components are spheres or tori.
	
	For the handlebody problem, Lott \cite{Lott} obtained the same genus restriction for complete metrics with nonnegative scalar curvature outside a compact set, under an additional quantitative area growth condition on surfaces exiting the end.
	
	Yan--Zhu subsequently replaced \eqref{eq:intro-BG} by different assumptions \cite{YanZhu}. They assume a lower Ricci curvature bound
	\begin{equation*}
		\Ric_g\geq Kg,
		\qquad K\leq0,
	\end{equation*}
	and, in the nonparabolic case, require the minimal positive Green function $G(p,\cdot)$ to vanish at infinity:
	\begin{equation}\label{eq:intro-Green-decay}
		G(p,x)\longrightarrow0
		\qquad\text{as }d_g(p,x)\longrightarrow\infty.
	\end{equation}
	Their proof uses harmonic level set exhaustions. In the parabolic case, a proper harmonic barrier on the end is combined
	with the Cheng--Yau gradient estimate \cite{ChengYau}, and the Ricci
	lower bound enters precisely in that estimate. In the nonparabolic case, \eqref{eq:intro-Green-decay} makes the
	Green-distance function $b=G^{-1}$ proper, allowing
	$b$ formulation of the scalar curvature level set estimates
	in Colding--Minicozzi \cite{CM} to be applied globally. A harmonic level set approach to scalar
	curvature in dimension three was developed by Stern through his
	harmonic map identity \cite{Stern}. Green function
	monotonicity formulas under nonnegative Ricci curvature were developed by
	Colding \cite{ColdingMonotonicity} and Colding--Minicozzi
	\cite{ColdingMinicozziHarmonic}, while the scalar curvature Green function
	approach was subsequently developed by Munteanu--Wang
	\cite{MunteanuWangComparison,MunteanuWangGeometry}.

	We prove both statements without any auxiliary hypothesis.
	
	\begin{theorem}\label{thm:contractible-main}
		Let $(M^3,g)$ be a contractible complete Riemannian $3$-manifold satisfying $R_g\geq0.$
		Then $M$ is diffeomorphic to $\R^3$.
	\end{theorem}
	
	\begin{theorem}\label{thm:handlebody-main}
		Let $M_\gamma$ be the interior of a compact handlebody of genus $\gamma$. If $M_\gamma$ admits a complete Riemannian metric $g$ satisfying
		$
		R_g\geq0,
		$
		then
		$
		\gamma\leq1.
		$
	\end{theorem}
	
	Throughout, manifolds are smooth, connected, and without boundary unless
	stated otherwise. Homology is taken with coefficients in $\Q$, and
	$b_1(M):=\dim_{\Q}H_1(M;\Q)$ denotes the first Betti number of $M$.
	A complete manifold is called \emph{nonparabolic} if it admits a minimal
	positive Green function.
	
	The argument has three components: a topological reduction to the
	low-genus separator property $\mathrm{(LG)}$; an analytic theorem showing
	that a given complete, one-ended, nonparabolic metric with nonnegative
	scalar curvature and finite first Betti number forces $\mathrm{(LG)}$;
	and a metric replacement theorem reducing the general nonflat case to
	this nonparabolic setting. We begin with the topological reduction.
	
	\begin{definition}\label{def:admissible}
		Let $M$ be an orientable open $3$-manifold. A domain $U\Subset M$ is called an \emph{admissible separator} if
		\begin{enumerate}[label=\textup{(\roman*)}]
			\item $U$ has smooth compact boundary;
			\item $U$ is connected;
			\item $M\setminus U$ is connected;
			\item $\partial U$ is connected.
		\end{enumerate}
	\end{definition}
	
	\begin{definition}[Low-genus separator property]\label{def:low-genus-property}
		An orientable open $3$-manifold $M$ is said to satisfy property $\mathrm{(LG)}$ if every compact set $K\Subset M$ is contained in an admissible separator $U$ such that
		$$
		\genus(\partial U)\leq1.
		$$
	\end{definition}
	
	The following result is precisely a criterion for property $\mathrm{(LG)}$.
	
	\begin{theorem}\label{thm:low-genus}
		Let $(M^3,g)$ be a complete, orientable, one-ended, nonparabolic Riemannian manifold satisfying
		$$
		R_g\geq0,
		\qquad
		b_1(M)<\infty.
		$$
		Then $M$ satisfies property $\mathrm{(LG)}$.
	\end{theorem}
	
	The topological result is summarized by the following reduction.
	
	\begin{proposition}[Topological reduction]\label{prop:topological-reduction}
		The following statements hold.
		\begin{enumerate}[label=\textup{(\roman*)}]
			\item Let $M^3$ be contractible and suppose that $M$ admits a complete metric of nonnegative scalar curvature. If $M$ satisfies property $\mathrm{(LG)}$, then $M$ is diffeomorphic to $\R^3$.
			\item If the open handlebody $M_\gamma$ satisfies property $\mathrm{(LG)}$, then $\gamma\leq1$.
		\end{enumerate}
	\end{proposition}
	
	Combining Theorem~\ref{thm:low-genus} with Proposition~\ref{prop:topological-reduction} gives the two nonparabolic consequences below.
	
	\begin{corollary}\label{cor:nonparabolic-contractible}
		Let $(M^3,g)$ be a contractible complete nonparabolic Riemannian $3$-manifold with $R_g\geq0$. Then $M$ is diffeomorphic to $\R^3$.
	\end{corollary}
	
	\begin{corollary}\label{cor:nonparabolic-handlebody}
		Let $M_\gamma$ be the interior of a compact handlebody of genus $\gamma$. If $M_\gamma$ admits a complete nonparabolic metric $g$ with $R_g\geq0$, then
		$
		\gamma\leq1.
		$
	\end{corollary}
	
	To pass from the parabolic to the nonparabolic setting, we use a conformal
	deformation driven by a boundary-normalized exterior Evans potential, namely,
	a positive harmonic function on an exterior domain that vanishes on its
	boundary and tends to $+\infty$ at infinity.
	
	\begin{proposition}
		\label{prop:nonparabolicization}
		Let $M^3$ be one-ended. If $M$ admits a complete Riemannian metric $g$ with
		$
		R_g>0,
		$
		then $M$ admits a complete nonparabolic Riemannian metric $\widehat g$ with
		$
		R_{\widehat g}>0.
		$
	\end{proposition}
	
	Together with a Kazdan-type deformation to positive scalar curvature
	\cite[Theorem A]{Kazdan}, Proposition~\ref{prop:nonparabolicization}
	yields the following metric replacement theorem.
	
	\begin{theorem}\label{thm:metric-replacement}
		Let $(M^3,g)$ be complete and one-ended. Assume that
		$$
		R_g\geq0,
		\qquad
		\Ric_g\not\equiv0.
		$$
		Then $M$ admits a complete nonparabolic metric $\widehat g$ satisfying
		$
		R_{\widehat g}>0.
		$
	\end{theorem}
	
	We sketch the proof of Theorem~\ref{thm:low-genus}. Suppose, toward a contradiction, that property $\mathrm{(LG)}$ fails. Then some compact set $K_0$ is contained only in admissible separators of genus at least two. We construct a smooth proper Morse function $\rho$ with globally bounded gradient and form the nested filled sublevel domains $D_\tau$. On an interval containing no critical value, the boundary moves outward with normal speed
	$$
	V_\tau=\frac{1}{|\nabla\rho|}\geq L^{-1}.
	$$
	For the Dirichlet Green function $G_\tau$ of $D_\tau$, let $m_\tau$ be its Robin constant and put
	$$
	q_\tau:=-\partial_{\nu_\tau}G_\tau>0
	\qquad\text{on }\partial D_\tau.
	$$
	Hadamard's variational formula gives
	\begin{equation}\label{eq:intro-Hadamard}
		\frac{d}{d\tau}m_\tau
		=\frac{1}{4\pi}\int_{\partial D_\tau}
		\frac{q_\tau^2}{|\nabla\rho|}\dd A_g
		\geq\frac{1}{4\pi L}\int_{\partial D_\tau}q_\tau^2\dd A_g.
	\end{equation}
	Thus the theorem reduces to establishing a
	positive lower bound, uniform in $\tau$, for $\int_{\partial D_\tau}q_\tau^2\dd A_g$.
	
	For each fixed outer domain, set $b_\tau=G_\tau^{-1}$ and consider
	$$
	A_\tau(r):=r^{-2}\int_{\{b_\tau=r\}}|\nabla b_\tau|^2\dd A_g.
	$$
	A fixed initial Dirichlet domain and domain monotonicity provide a threshold $r_*$, independent of $\tau$, beyond which the Green superlevels contain $K_0$ and the fixed compact set from Lemma~\ref{lem:compact-detector} whose inclusion forces boundary connectedness. Hence every sufficiently large regular level is an admissible separator of genus at least two. The resulting Gauss--Bonnet gap is inserted into a critical level
	Green function identity. The local scalar curvature calculation
	underlying this identity is the harmonic level set mechanism appearing
	in Stern's work \cite{Stern} and in the Green function estimates of
	Munteanu--Wang
	\cite{MunteanuWangComparison,MunteanuWangGeometry}; we express it in
	the $b=G^{-1}$ formalism of Colding--Minicozzi \cite{CM}. We show that no defect measure occurs at critical Green levels, so the identity can be integrated across them. An integral Riccati
	comparison then gives
	\begin{equation*}
		A_\tau(r)\geq c_*r^2
		\qquad\text{for all sufficiently large }r,
	\end{equation*}
	where $c_*>0$ is independent of the outer parameter. Boundary regularity then yields
	$$
	\lim_{r\to\infty}\frac{A_\tau(r)}{r^2}
	=\int_{\partial D_\tau}q_\tau^2\dd A_g,
	$$
	and therefore the uniform lower bound required in \eqref{eq:intro-Hadamard}.
	
	Two kinds of critical values have to be treated separately. Critical values of \(\tau\) may change the topology of \(D_\tau\); domain monotonicity shows that the corresponding jumps of the Robin constant are nonnegative. Critical values of the inner Green parameter $r$ may produce singular level sets, but local semiconvexity, the codimension-two estimate for the critical set, and a $BV$ argument show that the level set identity acquires no defect measure. Integrating \eqref{eq:intro-Hadamard}, including the nonnegative outer jumps, forces $m_\tau$ to diverge. This contradicts the uniform upper bound supplied by the global minimal Green function. 
	
	Theorem \ref{thm:metric-replacement} has two components. If $R_g\geq0$ and $\Ric_g\not\equiv0$, Kazdan's deformation theorem produces a complete metric $g_+$, uniformly equivalent to $g$, with $R_{g_+}>0$. If $g_+$ is parabolic, a proper exterior Evans potential is used to construct a conformal factor that preserves strict scalar curvature positivity, preserves completeness, and makes a compact set have positive $\widehat g$-capacity. The resulting metric is nonparabolic. If instead $\Ric_g\equiv0$, then the metric is flat in dimension three; the contractible case is Euclidean, while in the handlebody case a deck transformation packing argument excludes the free group $F_\gamma$ for $\gamma\geq2$.

	The paper is organized as follows. Section~\ref{sec:topological-reduction} isolates the topological reduction to property $\mathrm{(LG)}$. Subsection~\ref{sec:morse-exhaustion} constructs the bounded-gradient filled Morse exhaustion. Subsection~\ref{sec:domain-variation} develops domain monotonicity and the Hadamard formula for Dirichlet Green functions and Robin constants. Subsection~\ref{sec:green-levels} treats Dirichlet Green levels. Subsection~\ref{sec:low-genus-proof} proves Theorem~\ref{thm:low-genus} and its nonparabolic corollaries. Section~\ref{sec:metric-replacement} proves Theorem \ref{thm:metric-replacement}. Section~\ref{sec:main-theorems} treats the flat branch and completes the proofs of Theorems~\ref{thm:contractible-main} and~\ref{thm:handlebody-main}.
	
	\section{Topological reduction to low-genus separators}\label{sec:topological-reduction}
	
	This section isolates the topological result from the analytic argument. The first two results control the number and connectedness of boundary components; they will also be used for Dirichlet Green superlevels in Subsection~\ref{sec:green-levels}. The final two lemmas identify the compact topological obstructions that property $\mathrm{(LG)}$ must exclude.
	
	\subsection{Boundary components}
	
	The following elementary estimate is needed because small Dirichlet Green levels need not be connected.
	
	\begin{lemma}\label{lem:boundary-components}
		Let $M$ be a connected manifold and let $U\Subset M$ be a connected smooth domain such that $M\setminus U$ is connected. If
		$$
		\partial U=\Sigma_1\sqcup\cdots\sqcup\Sigma_N,
		$$
		then
		\begin{equation}\label{eq:component-bound}
			N-1\leq b_1(M).
		\end{equation}
	\end{lemma}
	
	\begin{proof}
		Choose a two-sided open collar $\mathcal N\cong\partial U\times(-\eps,\eps)$. Let $A$ be a connected open thickening of $U$ obtained by adjoining the collar, and let $B$ be a connected open thickening of $M\setminus U$ using the same collar. Then
		$$
		M=A\cup B,
		\qquad
		A\cap B\simeq\partial U.
		$$
		The connectedness assumptions give
		$$
		\widetilde H_0(A;\Q)=0,
		\qquad
		\widetilde H_0(B;\Q)=0,
		$$
		whereas
		$$
		\widetilde H_0(A\cap B;\Q)
		\cong \widetilde H_0(\partial U;\Q)
		\cong \Q^{N-1}.
		$$
		The reduced Mayer--Vietoris sequence contains the exact segment
		$$
		H_1(M;\Q)
		\longrightarrow
		\widetilde H_0(A\cap B;\Q)
		\longrightarrow
		\widetilde H_0(A;\Q)\oplus\widetilde H_0(B;\Q)=0.
		$$
		Thus the first arrow is surjective, and \eqref{eq:component-bound} follows by taking dimensions.
	\end{proof}
	
	\subsection{A compact set detecting disconnected boundaries}
	
	We also use the following topological observation, stated in essentially this form as \cite[Lemma 2.6]{CLX}.
	
	\begin{lemma}\label{lem:compact-detector}
		Let $M$ be an orientable, one-ended manifold with $b_1(M)<\infty$. There exists a compact set $S\Subset M$ with the following property. If $U\Subset M$ is a connected $C^1$ domain and $M\setminus U$ has no precompact connected component, then either $\partial U$ is connected or every connected component of $\partial U$ intersects $S$.
	\end{lemma}
	
	\begin{proof}
		Choose smooth closed loops $\gamma_1,\ldots,\gamma_\beta$ whose homology classes form a basis of $H_1(M;\Q)$, where $\beta=b_1(M)$, and choose a compact set $S$ containing their images.
		
		First, $M\setminus U$ is connected. Indeed, if it had two distinct nonprecompact components, each would contain a proper ray. Since $\overline U$ is compact, the two rays would determine distinct ends of $M$, contradicting one-endedness.
		
		Suppose that $\partial U$ is disconnected, and let $\Sigma$ be one of its connected components. Since $M$ is orientable, the compact two-sided surface $\Sigma$ is cooriented; its compactly supported Poincar\'e dual makes algebraic intersection with $\Sigma$ a homology-invariant functional. Since both $U$ and $M\setminus U$ are connected and there is at least one other boundary component, a small oriented normal arc through $\Sigma$ can be completed to a closed loop avoiding $\Sigma$ except at that single transverse crossing: connect its endpoint in $U$ to another boundary component inside $U$, cross that component, and return to the exterior endpoint inside $M\setminus U$. This loop has intersection number $\pm1$, so the functional
		$$
		I_\Sigma:H_1(M;\Q)\longrightarrow\Q
		$$
		is nonzero. If $\Sigma\cap S=\varnothing$, then every $\gamma_i$ is disjoint from $\Sigma$, so $I_\Sigma([\gamma_i])=0$ for all $i$. Since the classes $[\gamma_i]$ span $H_1(M;\Q)$, this would imply $I_\Sigma=0$, a contradiction. Hence every component of a disconnected $\partial U$ intersects $S$.
	\end{proof}

	\subsection{Open handlebodies}
	
	\begin{lemma}\label{lem:handlebody-core}
		Let $M_\gamma$ be the interior of a compact orientable handlebody of genus $\gamma$. There exists a compact set $K_\gamma\Subset M_\gamma$ such that every admissible separator $U$ with $K_\gamma\subset U$ satisfies
		\begin{equation}\label{eq:handlebody-genus-lower}
			\genus(\partial U)\geq\gamma.
		\end{equation}
	\end{lemma}
	
	\begin{proof}
		Choose a compact core $K_\gamma$ so that
		$$
		M_\gamma\setminus K_\gamma\cong\Sigma_\gamma\times(0,\infty),
		$$
		where $\Sigma_\gamma$ is the closed orientable surface of genus $\gamma$. If $U$ contains $K_\gamma$, then $\partial U$ lies in the collar and separates its compact end from infinity. Choose $T$ so large that $\overline U$ lies below $\Sigma_\gamma\times\{T\}$. The compact region
		$$
		W_T:=(M_\gamma\setminus U)\cap\bigl(\Sigma_\gamma\times(0,T]\bigr)
		$$
		is an oriented cobordism from $\partial U$ to $\Sigma_\gamma\times\{T\}$. Hence
		$$
		[\partial U]=[\Sigma_\gamma\times\{T\}]
		\quad\text{in }H_2(\Sigma_\gamma\times(0,\infty);\Z),
		$$
		so projection onto $\Sigma_\gamma$ restricts to a map
		$$
		f:\partial U\longrightarrow\Sigma_\gamma
		$$
		of degree $\pm1$.
		
		The induced map
		$$
		f^*:H^1(\Sigma_\gamma;\Q)\longrightarrow H^1(\partial U;\Q)
		$$
		is injective. Indeed, if $f^*\alpha=0$, then for every $\beta\in H^1(\Sigma_\gamma;\Q)$,
		$$
		\deg(f)\,\langle\alpha\smile\beta,[\Sigma_\gamma]\rangle
		=\langle f^*(\alpha\smile\beta),[\partial U]\rangle=0,
		$$
		and Poincar\'e duality on $\Sigma_\gamma$ gives $\alpha=0$. Therefore
		$$
		2\gamma=\dim H^1(\Sigma_\gamma;\Q)
		\leq\dim H^1(\partial U;\Q)
		=2\genus(\partial U),
		$$
		which proves \eqref{eq:handlebody-genus-lower}. This is the fixed-core form of the argument in \cite[Lemma 2.7]{CLX}.
	\end{proof}

	\subsection{Contractible manifolds}
	
	\begin{lemma}\label{lem:contractible-core}
		Let $M^3$ be contractible and suppose that $M$ admits a complete metric of nonnegative scalar curvature. If $M$ is not diffeomorphic to $\R^3$, then there exists a compact set $K\Subset M$ such that every admissible separator containing $K$ has genus at least two.
	\end{lemma}
	
	\begin{proof}
		Suppose no such compact set exists. Let $L_1\Subset L_2\Subset\cdots$ be a compact exhaustion of $M$. Inductively choose admissible separators $U_j$ such that
		\begin{equation*}
			L_j\cup\overline{U_{j-1}}\subset U_j,
			\qquad
			\genus(\partial U_j)\leq1,
		\end{equation*}
		with $U_0=\varnothing$. Then $\{U_j\}$ is a smooth nested exhaustion with connected boundaries and connected complements. Passing to an infinite subsequence, all boundaries are spheres or all are tori; the subsequence remains a nested exhaustion.
		
		If all boundaries are spheres, the exhaustion theorem of Husch--Price implies that $M$ is homeomorphic to $\R^3$; see \cite{HuschPrice,HuschPriceAddendum} and \cite[Lemma 2.3]{CLX}. If all boundaries are tori, the nonnegative-scalar-curvature torus-exhaustion rigidity theorem of Wang, in the form recorded in \cite[Lemma 2.4]{YanZhu}, likewise implies that $M$ is homeomorphic to $\R^3$; see \cite{WangI,WangII}. By Moise's uniqueness theorem for smooth structures on $3$-manifolds \cite{Moise}, either homeomorphism type is represented by a diffeomorphism $M\cong_{\mathrm{diff}}\R^3$. Both alternatives contradict the hypothesis.
	\end{proof}

	\subsection{Proof of the topological reduction}
	
	\begin{proof}[Proof of Proposition~\ref{prop:topological-reduction}]
		For part \textup{(i)}, suppose that $M$ is not diffeomorphic to $\R^3$. Lemma~\ref{lem:contractible-core} provides a compact set $K\Subset M$ such that every admissible separator containing $K$ has genus at least two. This contradicts property $\mathrm{(LG)}$.
		
		For part \textup{(ii)}, let $K_\gamma$ be the compact core from Lemma~\ref{lem:handlebody-core}. Property $\mathrm{(LG)}$ gives an admissible separator $U$ containing $K_\gamma$ with $\genus(\partial U)\leq1$, whereas Lemma~\ref{lem:handlebody-core} gives $\genus(\partial U)\geq\gamma$. Hence $\gamma\leq1$.
	\end{proof}

	\section{Morse exhaustions and domain variation}
	\subsection{A bounded gradient filled Morse exhaustion}\label{sec:morse-exhaustion}
	We will construct a nested exhaustion whose regular boundaries move with a uniform positive lower bound on their outward normal speed. This is the outer family used in the Robin constant variation.

	\begin{proposition}\label{prop:morse-function}
		Let $(M,g)$ be complete and let $p\in M$. There exists a smooth proper Morse function
		$
		\rho:M\longrightarrow[0,\infty)
		$
		and a finite constant $L$ such that
		\begin{equation}\label{eq:rho-gradient}
			|\nabla\rho|\leq L
			\qquad\text{on }M.
		\end{equation}
		Moreover, every compact interval contains only finitely many critical values of $\rho$.
	\end{proposition}
	
	\begin{proof}
		The distance function $r_p(x)=d_g(p,x)$ is proper by Hopf--Rinow and is $1$-Lipschitz. By the smooth Lipschitz approximation theorem of Azagra--Ferrera--L\'opez-Mesas--Rangel \cite[Theorem 1]{AzagraEtAl}, there is a smooth function $f$ such that
		$$
		|f-r_p|<1,
		\qquad
		\operatorname{Lip}(f)<2.
		$$
		In particular, $f$ is proper and $|\nabla f|<2$.
		
		Morse functions are dense in the strong Whitney $C^\infty$ topology by jet transversality; see, for example, \cite[Chapter 6]{Hirsch}. We may therefore choose a Morse function $\rho_0$ in a strong neighborhood of $f$ satisfying
		$$
		|\rho_0-f|<1,
		\qquad
		|d\rho_0-df|_g<1
		$$
		at every point. Then
		$$
		|\rho_0-r_p|<2,
		\qquad
		|\nabla\rho_0|<3.
		$$
		Thus $\rho_0$ is proper and bounded from below. Adding a constant gives a map $\rho:M\to[0,\infty)$ satisfying \eqref{eq:rho-gradient}, for instance with $L=3$.
		
		The critical set of a Morse function is discrete and closed. Since $\rho^{-1}([a,b])$ is compact for every bounded interval $[a,b]$, it contains only finitely many critical points and hence only finitely many critical values.
	\end{proof}

	Assume from now on that $M$ has one end. Fix a function $\rho$ as in Proposition~\ref{prop:morse-function} and a pole $p\in M$. For $\tau>\rho(p)$, let $C_\tau$ be the connected component of $\{\rho<\tau\}$ containing $p$. Since $C_\tau$ is precompact and $M$ has one end, $M\setminus C_\tau$ has a unique nonprecompact connected component; denote it by $E_\tau$. Define the filled sublevel domain
	\begin{equation*}
		D_\tau:=M\setminus E_\tau.
	\end{equation*}
	Equivalently, $D_\tau$ is obtained from $C_\tau$ by adjoining all precompact complementary components. We will use the following properties of the filled exhaustion.
	
	\begin{lemma}\label{lem:filled-properties}
		The family $\{D_\tau\}_{\tau>\rho(p)}$ has the following properties.
		\begin{enumerate}[label=\textup{(\roman*)}]
			\item If $\tau$ is a regular value, then $D_\tau$ is a connected precompact smooth domain, $M\setminus D_\tau$ is connected, and every component of $\partial D_\tau$ is a component of $\rho^{-1}(\tau)$.
			\item If $\sigma<\tau$, then $D_\sigma\subset D_\tau$.
			\item The regular members exhaust $M$:
			\begin{equation}\label{eq:filled-exhaustion}
				\bigcup_{\substack{\tau>\rho(p)\\ \tau\ \mathrm{regular}}}D_\tau=M.
			\end{equation}
			\item If $\tau$ is regular, then along $\partial D_\tau$ the outward unit normal is
			\begin{equation}\label{eq:rho-normal}
				\nu_\tau=\frac{\nabla\rho}{|\nabla\rho|}.
			\end{equation}
		\end{enumerate}
	\end{lemma}
	
	\begin{proof}
		Fix first a regular value $\tau$. The component $C_\tau$ is connected and precompact, and $\partial C_\tau$ is a finite union of components of the compact regular level $\rho^{-1}(\tau)$. Hence $M\setminus C_\tau$ has only finitely many connected components: every such component is locally adjacent to at least one component of $\partial C_\tau$. Exactly one is nonprecompact, namely $E_\tau$, because $M$ has one end. Filling all the other components produces the connected open set $D_\tau$; its complement is $E_\tau$ and is connected.
		
		For completeness, we verify precompactness without assuming a priori that all filled pieces lie in a common compact set. Choose a connected smooth compact domain $Q_\tau$ with $C_\tau\subset\Int Q_\tau$. The complement of $Q_\tau$ has finitely many components and, by one-endedness, exactly one nonprecompact component. Adjoin the closures of all remaining components to $Q_\tau$ and denote the resulting compact set by $K_\tau$. Then $M\setminus K_\tau$ is connected, nonprecompact, and disjoint from $C_\tau$; it must therefore be contained in $E_\tau$. Consequently,
		$$
		D_\tau=M\setminus E_\tau\subset K_\tau,
		$$
		so $D_\tau\Subset M$. This proves (i).
		
		If $\sigma<\tau$, then $C_\sigma\subset C_\tau$. The connected nonprecompact set $E_\tau$ is disjoint from $C_\sigma$ and hence is contained in the unique nonprecompact component $E_\sigma$ of $M\setminus C_\sigma$. Thus $E_\tau\subset E_\sigma$ and $D_\sigma\subset D_\tau$. This argument does not require either parameter to be regular.
		
		For $x\in M$, choose a compact path from $p$ to $x$. Any regular value $\tau$ larger than the maximum of $\rho$ on that path satisfies $x\in C_\tau\subset D_\tau$. Regular values are unbounded because $\rho$ is proper and Morse. This proves \eqref{eq:filled-exhaustion}.
		
		Finally, at a regular value the boundary of $D_\tau$ consists precisely of the components of $\rho^{-1}(\tau)$ adjacent to $E_\tau$. Locally $D_\tau$ lies on the side $\rho<\tau$, so its outward normal is \eqref{eq:rho-normal}.
	\end{proof}
	
	\begin{lemma}\label{lem:smooth-motion}
		Let $\rho(p)<a<b$, and suppose that $[a,b]$ contains no critical value of
		$\rho$. Then there exists a smooth compactly supported ambient isotopy
		$\{\Phi_s\}_{0\leq s\leq b-a}$ such that
		$
		\Phi_s(D_a)=D_{a+s}.
		$
		The outward normal velocity of $\partial D_\tau$ is
		\begin{equation}\label{eq:velocity-lower}
			V_\tau=\frac{1}{|\nabla\rho|}\geq L^{-1}
			\qquad\text{on }\partial D_\tau.
		\end{equation}
	\end{lemma}
	
	\begin{proof}
		Choose $\delta>0$ such that
		$
		\rho(p)<a-\delta
		$
		and the compact slab
		$
		\rho^{-1}([a-\delta,b+\delta])
		$
		contains no critical point of $\rho$. Let
		$
		\chi\in C_c^\infty((a-\delta,b+\delta))
		$
		satisfy $\chi\equiv1$ on $[a,b]$. On the slab define
		$$
		Y:=\chi(\rho)\frac{\nabla\rho}{|\nabla\rho|^2},
		$$
		and extend $Y$ by zero to all of $M$. Since $\chi(\rho)$ vanishes near the
		boundary of the slab, $Y$ is a smooth compactly supported vector field.
		
		Let $\Phi_s$ be the flow of $Y$. If $h_s$ denotes the flow on $\mathbb R$
		generated by the vector field $\chi(t)\partial_t$, then
		\begin{equation}\label{eq:rho-flow-reparametrization}
			\rho(\Phi_s(x))=h_s(\rho(x)).
		\end{equation}
		Each $h_s$ is increasing, and, since $\chi=1$ on $[a,b]$,
		$$
		h_s(a)=a+s
		\qquad(0\leq s\leq b-a).
		$$
		It follows from \eqref{eq:rho-flow-reparametrization} that
		$$
		\Phi_s(\{\rho<a\})=\{\rho<a+s\}.
		$$
		Moreover, $\Phi_s$ fixes $p$, because $Y$ vanishes near $p$. Hence
		$
		\Phi_s(C_a)=C_{a+s}.
		$
		As $\Phi_s$ is the identity outside a compact set, it maps the unique
		nonprecompact component of $M\setminus C_a$ onto the unique nonprecompact
		component of $M\setminus C_{a+s}$. Thus
		$
		\Phi_s(E_a)=E_{a+s}.
		$
		Taking complements gives
		$
		\Phi_s(D_a)=D_{a+s}.
		$
		
		Finally, let $x\in\partial D_a$ and set $x_s:=\Phi_s(x)$. Then
		$x_s\in\partial D_{a+s}$ and
		$
		\rho(x_s)=a+s.
		$
		Differentiating in $s$ and using
		$
		\nu_{a+s}=\frac{\nabla\rho}{|\nabla\rho|}
		$
		from Lemma~\ref{lem:filled-properties}, we obtain
		$$
		1
		=\frac{d}{ds}\rho(x_s)
		=\langle\nabla\rho,\dot x_s\rangle
		=|\nabla\rho|\,
		\langle\dot x_s,\nu_{a+s}\rangle.
		$$
		Therefore the outward normal velocity is
		$$
		V_{a+s}
		=\langle\dot x_s,\nu_{a+s}\rangle
		=\frac{1}{|\nabla\rho|}.
		$$
		The bound $|\nabla\rho|\leq L$ from
		\eqref{eq:rho-gradient} now gives \eqref{eq:velocity-lower}.
	\end{proof}

	\subsection{Domain variation of Robin constants}\label{sec:domain-variation}
	
	We now convert the construction in Subsection~\ref{sec:morse-exhaustion} into monotonicity and variation statements for Dirichlet Green functions and their Robin constants.

	Let $D\Subset M$ be a smooth connected domain containing $p$. Its Dirichlet Green function $G_D=G_D(p,\cdot)$ is normalized by
	\begin{equation}\label{eq:green-normalization}
		\Delta_gG_D=-4\pi\delta_p\quad\text{in }D,
		\qquad
		G_D=0\quad\text{on }\partial D.
	\end{equation}
	Let $\varrho=d_g(p,\cdot)$. In geodesic normal coordinates at $p$, the standard local parametrix gives
	\begin{equation}\label{eq:robin-expansion}
		G_D(p,x)=\frac{1}{\varrho}+m_D(p)+\psi_D(x),
		\qquad
		|\nabla^j\psi_D|=O(\varrho^{1-j}),\quad j=0,1,2.
	\end{equation}
	Here $m_D(p)$ is the Robin constant at $p$; see the Riemannian parametrix construction in \cite{Aubin} and the isolated singularity analysis in \cite{GilbargSerrin}.
	
	\begin{lemma}[Domain monotonicity]\label{lem:domain-monotonicity}
		If $D_1\subset D_2$ are smooth connected domains containing $p$, then
		\begin{equation}\label{eq:green-monotonicity}
			G_{D_1}\leq G_{D_2}
			\qquad\text{on }D_1\setminus\{p\},
		\end{equation}
		and
		\begin{equation}\label{eq:robin-monotonicity}
			m_{D_1}(p)\leq m_{D_2}(p).
		\end{equation}
		If $M$ is nonparabolic and $G_M$ is its minimal positive Green function with the same normalization, then
		\begin{equation}\label{eq:global-robin-bound}
			m_D(p)\leq m_M(p)<\infty.
		\end{equation}
	\end{lemma}
	
	\begin{proof}
		The singularities of $G_{D_2}$ and $G_{D_1}$ cancel. Their difference is harmonic on $D_1$, extends smoothly across $p$, and is nonnegative on $\partial D_1$. The maximum principle proves \eqref{eq:green-monotonicity}; evaluating the Robin constants at $p$ gives \eqref{eq:robin-monotonicity}. The same comparison with $G_M$ gives \eqref{eq:global-robin-bound}. 
	\end{proof}

	For a smooth family of domains $D_\tau$ and a boundary point moving with outward normal speed $V_\tau$, put
	\begin{equation*}
		q_\tau:=-\partial_{\nu_\tau}G_{D_\tau}>0
		\qquad\text{on }\partial D_\tau.
	\end{equation*}
	The strict positivity follows from the Hopf boundary lemma.
	
	\begin{proposition}[Hadamard formula for the Robin constant]\label{prop:hadamard}
		Let $\{D_\tau\}_{\tau\in I}$ be a smooth family of smooth precompact domains containing $p$, and assume that the ambient metric and the pole are fixed. If the outward normal velocity is $V_\tau$, then
		\begin{equation}\label{eq:hadamard}
			\frac{d}{d\tau}m_{D_\tau}(p)
			=\frac{1}{4\pi}\int_{\partial D_\tau}V_\tau q_\tau^2\dd A_g.
		\end{equation}
		For the filled Morse family on a regular interval, this becomes
		\begin{equation}\label{eq:hadamard-morse}
			\frac{d}{d\tau}m_{D_\tau}(p)
			=\frac{1}{4\pi}\int_{\partial D_\tau}\frac{q_\tau^2}{|\nabla\rho|}\dd A_g
			\geq\frac{1}{4\pi L}\int_{\partial D_\tau}q_\tau^2\dd A_g.
		\end{equation}
	\end{proposition}
	
	\begin{proof}
		Fix $\tau_0\in I$. Choose a smooth isotopy identifying the domains near $\tau_0$ and equal to the identity in a neighborhood of $p$. Pulling the problem back to $D_{\tau_0}$, subtracting a fixed local parametrix at $p$, and applying parameter-dependent Schauder theory to the resulting uniformly elliptic Dirichlet problem shows that the regular part and the Green function depend differentiably on $\tau$. This is the standard analytic justification of Hadamard's formula; see, for example, \cite[Chapter 5]{HenrotPierre} or \cite{Schiffer}.
		
		Let
		$$
		\dot G=\left.\frac{d}{d\tau}\right|_{\tau=\tau_0}G_{D_\tau}
		$$
		be the Eulerian derivative at fixed interior points. Since the Laplace operator and the pole are fixed,
		$$
		\Delta_g\dot G=0\quad\text{in }D_{\tau_0},
		$$
		and the singular part at $p$ is independent of $\tau$. If $y(\tau)\in\partial D_\tau$ moves with normal velocity $V=V_{\tau_0}$, differentiating $G_{D_\tau}(y(\tau))=0$ gives
		\begin{equation}\label{eq:boundary-shape-derivative}
			\dot G=-V\partial_\nu G=Vq
			\qquad\text{on }\partial D_{\tau_0}.
		\end{equation}
		Apply Green's second identity to $\dot G$ and $G=G_{D_{\tau_0}}$ on $D_{\tau_0}\setminus B_g(p,\eps)$, and then let $\eps\downarrow0$. The inner-boundary term converges to $4\pi\dot G(p)$, where $\dot G(p)$ denotes the derivative of the regular part. Since $G=0$ on the outer boundary, one obtains
		$$
		4\pi\dot G(p)
		=-\int_{\partial D_{\tau_0}}\dot G\,\partial_\nu G\dd A_g
		=\int_{\partial D_{\tau_0}}Vq^2\dd A_g,
		$$
		where \eqref{eq:boundary-shape-derivative} was used in the last equality. This proves \eqref{eq:hadamard}. Equation \eqref{eq:hadamard-morse} follows from Lemma~\ref{lem:smooth-motion}.
	\end{proof}

	The filled domains can change discontinuously when $\tau$ crosses a critical value of $\rho$. 
	
	\begin{proposition}\label{prop:critical-jumps}
		Let $a<b$ be regular values of $\rho$. Denote by $c_1<\cdots<c_N$ the critical values in $(a,b)$. For each $c_j$, define the one-sided Robin limits along regular parameters by
		$$
		m(c_j-):=\lim_{\substack{\tau\uparrow c_j\\ \tau\ \mathrm{regular}}}m_{D_\tau}(p),
		\qquad
		m(c_j+):=\lim_{\substack{\tau\downarrow c_j\\ \tau\ \mathrm{regular}}}m_{D_\tau}(p).
		$$
		Then the jump
		\begin{equation*}
			J_{c_j}:=m(c_j+)-m(c_j-)
		\end{equation*}
		is nonnegative, and
		\begin{align}
			m_{D_b}(p)-m_{D_a}(p)
			&=\frac{1}{4\pi}\int_a^b\int_{\partial D_\tau}
			\frac{q_\tau^2}{|\nabla\rho|}\dd A_g\dd\tau
			+\sum_{j=1}^N J_{c_j},
			\label{eq:hadamard-with-jumps}\\
			&\geq\frac{1}{4\pi L}\int_a^b\int_{\partial D_\tau}q_\tau^2\dd A_g\dd\tau.
			\label{eq:hadamard-lower-integrated}
		\end{align}
		where the parameter integral is over the regular set.
	\end{proposition}
	
	\begin{proof}
		By Lemma~\ref{lem:filled-properties}, the domains are nested. Lemma~\ref{lem:domain-monotonicity} therefore shows that
		$$
		\tau\longmapsto m_{D_\tau}(p)
		$$
		is nondecreasing on the regular parameters. For every regular $\tau\in[a,b]$, nesting also gives $D_\tau\subset D_b$, and hence
		$$
		m_{D_\tau}(p)\leq m_{D_b}(p)<\infty
		$$
		by domain monotonicity. Thus the function is locally bounded above without any nonparabolicity assumption. Every one-sided limit in the statement consequently exists and is finite, and monotonicity gives $J_{c_j}\geq0$.
		
		Set $c_0=a$ and $c_{N+1}=b$. Fix $j$ and regular numbers
		$$
		c_j<\alpha<\beta<c_{j+1}.
		$$
		On $[\alpha,\beta]$ the family is smooth, so Proposition~\ref{prop:hadamard} gives
		\begin{equation*}
			m_{D_\beta}(p)-m_{D_\alpha}(p)
			=\frac{1}{4\pi}\int_\alpha^\beta\int_{\partial D_\tau}
			\frac{q_\tau^2}{|\nabla\rho|}\dd A_g\dd\tau.
		\end{equation*}
		The integrand is nonnegative. Letting $\alpha\downarrow c_j$ and $\beta\uparrow c_{j+1}$ through regular values and using monotone convergence yields
		\begin{align*}
			m(c_{j+1}-)-m(c_j+)
			=\frac{1}{4\pi}\int_{c_j}^{c_{j+1}}\int_{\partial D_\tau}
			\frac{q_\tau^2}{|\nabla\rho|}\dd A_g\dd\tau,
		\end{align*}
		with the evident conventions $m(c_0+)=m_{D_a}(p)$ and $m(c_{N+1}-)=m_{D_b}(p)$. In particular, the improper integrals at the critical endpoints are finite; there is no additional parameter-singular term on a regular interval.
		
		Summing these identities and inserting
		$$
		m(c_j+)-m(c_j-)=J_{c_j}
		$$
		gives \eqref{eq:hadamard-with-jumps}. Dropping the nonnegative jumps and using $|\nabla\rho|\leq L$ proves \eqref{eq:hadamard-lower-integrated}.
	\end{proof}

	By \eqref{eq:hadamard-lower-integrated}, the remaining analytic task is to prove a lower bound
	$$
	\int_{\partial D_\tau}q_\tau^2\dd A_g\geq c_*>0
	$$
	that is independent of the outer parameter $\tau$.
	
	\section{Green level analysis and the low-genus separator theorem}
	\subsection{Dirichlet Green levels}\label{sec:green-levels}
	
	Fixing one outer domain $D$, we now study the inner level sets of its Dirichlet Green-distance function. 
	
	Fix a smooth connected domain $D\Subset M$ containing $p$ and assume that $M\setminus D$ is connected. Let $G=G_D(p,\cdot)$ be normalized by \eqref{eq:green-normalization}, and define
	\begin{equation*}
		b:=G^{-1}.
	\end{equation*}
	Then $b\to0$ at the pole and $b\to\infty$ at $\partial D$. In particular, $b:D\setminus\{p\}\to(0,\infty)$ is proper. We use the convention $G(p)=+\infty$ and $b(p)=0$. For a regular value $r$, set
	\begin{equation*}
		\Sigma_r:=\{b=r\},
		\qquad
		U_r:=\{b<r\}=\{p\}\cup\{G>r^{-1}\}.
	\end{equation*}
	Thus $U_r$ is an open neighborhood of $p$ with smooth boundary $\Sigma_r$.
	
	\subsubsection{Connectivity of level components}
	
	Related connectivity and component-control results for proper Green
	functions and harmonic barrier functions appear in
	\cite[Lemma~2.3]{MunteanuWangComparison} and
	\cite[Lemma~2.2]{MunteanuWangGeometry}. We record here the bounded
	Dirichlet-domain form needed in the sequel.
	\begin{lemma}\label{lem:green-connectivity}
		For every regular value $r>0$, the set $U_r$ is connected and $M\setminus U_r$ is connected. Consequently, if $N(r)$ is the number of components of $\Sigma_r$, then
		\begin{equation}\label{eq:level-component-bound}
			N(r)-1\leq b_1(M).
		\end{equation}
	\end{lemma}
	
	\begin{proof}
		Suppose that $U_r$ has a component $C$ not containing $p$. A sufficiently small punctured ball about $p$ lies in the component containing $p$, so $C$ stays away from the pole. Since $G=0$ on $\partial D$, it also stays away from $\partial D$; hence $\overline C\Subset D\setminus\{p\}$. The harmonic function $G$ equals $r^{-1}$ on $\partial C$ and is strictly larger than $r^{-1}$ in $C$, contradicting the maximum principle. Hence $U_r$ is connected.
		
		Let
		$$
		\Omega_-:=\{x\in D:G(x)<r^{-1}\},
		$$
		and let $W$ be a connected component of $\Omega_-$. If $\overline W\cap\partial D=\varnothing$, then $W\Subset D\setminus\{p\}$, the boundary value of $G$ on $\partial W$ is $r^{-1}$, and $G<r^{-1}$ in $W$, contradicting the minimum principle. Thus every $\overline W$ meets $\partial D$.
		
		We also claim that every point of the level $\{G=r^{-1}\}$ lies in the closure of $\Omega_-$. Otherwise $G\geq r^{-1}$ in a neighborhood of such a point, where equality is attained. This would make the point a local minimum of the nonconstant harmonic function $G$, contrary to the strong minimum principle. It follows that
		\begin{equation}\label{eq:green-complement-union}
			M\setminus U_r
			=(M\setminus D)\cup\bigcup_{W\in\pi_0(\Omega_-)}\overline W.
		\end{equation}
		The set $M\setminus D$ is connected by hypothesis, and each connected set $\overline W$ meets $\partial D\subset M\setminus D$. Therefore the union in \eqref{eq:green-complement-union} is connected.
		
		The estimate \eqref{eq:level-component-bound} now follows by applying Lemma~\ref{lem:boundary-components} to $U_r$.
	\end{proof}
	
	\begin{corollary}\label{cor:large-level-connected}
		Let $S$ be the compact set from Lemma~\ref{lem:compact-detector}. If $S\subset U_r$, then $\Sigma_r$ is connected.
	\end{corollary}
	
	\begin{proof}
		By Lemma~\ref{lem:green-connectivity}, $U_r$ and its complement are connected. Since $S\subset U_r$, Lemma~\ref{lem:compact-detector} applies.
	\end{proof}
	
	\subsubsection{The Colding--Minicozzi identity}
	
	On the regular set of $b$, write
	$$
	\nu:=\frac{\nabla b}{|\nabla b|},
	\qquad
	B:=\Hess(b^2)-2|\nabla b|^2g.
	$$
	For regular $r$, define
	\begin{align*}
		A(r)&:=r^{-2}\int_{\Sigma_r}|\nabla b|^2\dd A_g,\\
		B_1(r)&:=r^{-2}\int_{\Sigma_r}\Hess(b^2)(\nu,\nu)\dd A_g,\\
		B_2(r)&:=\int_{\Sigma_r}\frac{|B|^2}{|\nabla b^2|^2}\dd A_g,\\
		S_1(r)&:=\int_{\Sigma_r}R_g\dd A_g,\\
		\kappa(r)&:=\int_{\Sigma_r}K_{\Sigma_r}\dd A_g.
	\end{align*}
	Here $K_{\Sigma_r}$ is the intrinsic Gauss curvature, and all integrals are summed over the components of $\Sigma_r$.
	
	\begin{lemma}\label{lem:CM-localization}
		The regular-value functions $A$, $B_1$, and $S_1$ defined above admit continuous representatives on $(0,\infty)$. The function $A$ is locally Lipschitz and absolutely continuous, and for almost every $r>0$ one has
		\begin{equation}\label{eq:CM-ae-first-variation}
			rA'(r)=\frac12B_1(r)-A(r).
		\end{equation}
	\end{lemma}
	
	\begin{proof}
		Regard $D$ as an open manifold. The function $G:D\setminus\{p\}\to(0,\infty)$ is positive and harmonic, and every compact $b$-band is relatively compact in $D\setminus\{p\}$. Indeed, for $0<\alpha<\beta<\infty$ the set
		$$
		\{\alpha\leq b\leq\beta\}
		=G^{-1}([\beta^{-1},\alpha^{-1}])
		$$
		stays away from both $p$ and $\partial D$. The pole expansion \eqref{eq:robin-expansion} is the Green-function asymptotic used in \cite{CM}.
		
		Fix $[\alpha,\beta]\Subset(0,\infty)$ and choose
		$$
		0<\alpha'<\alpha<\beta<\beta'<\infty
		$$
		so that $\{\alpha'\leq b\leq\beta'\}\Subset D\setminus\{p\}$. Let $\chi=\chi(b)$ be supported in this larger band and identically one on a neighborhood of $\{\alpha\leq b\leq\beta\}$. The coarea, divergence, and first-variation arguments in \cite[Appendix~A]{CM} use only the geometry of the relevant compact band. Multiplying the test fields there by $\chi$ makes them compactly supported in $D\setminus\{p\}$, while every term containing $d\chi$ is supported outside the levels $r\in[\alpha,\beta]$. Thus the continuity and absolute-continuity conclusions of that appendix, together with \cite[Lemma 1.9]{CM}, apply without an additional boundary term and give \eqref{eq:CM-ae-first-variation}. Since $[\alpha,\beta]$ was arbitrary, the assertions hold on $(0,\infty)$. 
	\end{proof}
	
	Stern's argument passes through critical values by regularizing the
	gradient norm and shrinking a neighborhood of the critical values;
	see \cite[\S2]{Stern}. For the argument below, we need an
	exact distributional identity on every compact $b$-band. We therefore
	verify that the relevant weighted divergence has no singular part
	supported on the critical set.
	\begin{lemma}[No critical-set defect]\label{lem:positive-defect}
		Set
		$$
		W:=|\nabla b^2|,
		\qquad
		Z:=\{\nabla b=0\}.
		$$
		Then
		$$
		\Vol_g(Z)=0,
		\qquad
		W\in W^{2,1}_{\mathrm{loc}}(D\setminus\{p\}).
		$$
		Consequently, if $\omega$ is $C^1$ on a neighborhood of a compact
		$b$-band, then
		\begin{equation}\label{eq:weighted-divergence-measure}
			\operatorname{div}_g(\omega\nabla W)=h_\omega
			\quad\text{in }\mathcal D',
			\qquad
			h_\omega
			:=\omega\Delta_gW+\langle\nabla\omega,\nabla W\rangle
			\in L^1_{\mathrm{loc}}.
		\end{equation}
		$h_\omega$ agrees almost everywhere on $\{\nabla b\neq0\}$
		with the classical divergence. If $0<\alpha<\beta$ are regular values,
		then
		\begin{align}
			&\int_{\Sigma_\beta}
			\omega\langle\nabla W,\nu\rangle\dd A_g
			-\int_{\Sigma_\alpha}
			\omega\langle\nabla W,\nu\rangle\dd A_g \notag\\
			&\qquad=
			\int_{\{\alpha<b<\beta\}}h_\omega\dd V_g.
			\label{eq:weighted-Gauss-Green}
		\end{align}
	\end{lemma}
	
	\begin{proof}
		Put $f:=b^2$. Fix a smooth coordinate ball
		$U\Subset D\setminus\{p\}$, shrink it so that its coordinate image is
		convex, and choose a smooth $g$-orthonormal frame
		$e_1,e_2,e_3$ on $U$. Writing
		$$
		Y:=(e_1f,e_2f,e_3f),
		\qquad
		W=|Y|,
		$$
		we have $Y\in C^2(U;\mathbb R^3)$. Hence, on every relatively compact
		subball of $U$,
		$$
		|Y(x+h)+Y(x-h)-2Y(x)|\leq C|h|^2.
		$$
		By the triangle inequality,
		$$
		W(x+h)+W(x-h)-2W(x)\geq-C|h|^2.
		$$
		Thus $W$ is locally semiconvex in coordinate variables. In particular,
		$$
		W\in W^{1,\infty}_{\mathrm{loc}}(U),
		\qquad
		\partial_iW\in BV_{\mathrm{loc}}(U);
		$$
		see \cite[Chapters 2--3]{CannarsaSinestrari} and
		\cite[Chapter 3]{AFP}.
		
		We next examine the possible singular part of the Hessian. Since
		$$
		\nabla b=-G^{-2}\nabla G,
		$$
		the critical set $Z$ agrees with $\{\nabla G=0\}$. Covering a compact
		$b$-band by finitely many harmonic-coordinate charts, we have
		$$
		g^{ij}\partial_{ij}G=0
		$$
		with smooth uniformly elliptic coefficients. It follows from
		\cite[Theorem 1.1]{HardtEtAl} that
		$$
		\mathcal H^1\bigl(Z\cap\{\alpha\leq b\leq\beta\}\bigr)<\infty
		$$
		on every compact band. In particular,
		$$
		\mathcal H^2(Z)=0,
		\qquad
		\Vol_g(Z)=0
		$$
		locally in $D\setminus\{p\}$.
		
		In a coordinate chart, set $v_i:=\partial_iW$. Its $BV$ derivative
		decomposes as
		$$
		Dv_i=\nabla v_i\,\mathcal L^3+D^jv_i+D^cv_i.
		$$
		The function $W$ is smooth on $U\setminus Z$, so both singular terms are
		concentrated on $Z$. The jump part satisfies
		$$
		|D^jv_i|\ll\mathcal H^2,
		$$
		whereas the Cantor part vanishes on every $\sigma$-finite
		$\mathcal H^2$ set; see
		\cite[Theorem 3.78 and Proposition 3.92]{AFP}. Since
		$\mathcal H^2(Z)=0$, both terms vanish:
		$$
		D^jv_i=D^cv_i=0.
		$$
		Therefore \(Dv_i\) is absolutely continuous with an \(L^1_{\mathrm{loc}}\)
		density, and hence
		$$
		W\in W^{2,1}_{\mathrm{loc}}(D\setminus\{p\}).
		$$
		Passing from coordinate to covariant second derivatives only adds terms
		consisting of smooth Christoffel symbols multiplied by
		$DW\in L^\infty_{\mathrm{loc}}$. Thus the weak Laplacian
		$\Delta_gW$ belongs to \(L^1_{\mathrm{loc}}\). Since \(W\) is smooth on
		$\{\nabla b\neq0\}\), this weak Laplacian agrees there almost everywhere
		with the classical one.
		
		The weak product rule now gives
		$$
		\operatorname{div}_g(\omega\nabla W)
		=\omega\Delta_gW+\langle\nabla\omega,\nabla W\rangle
		=h_\omega,
		$$
		which proves \eqref{eq:weighted-divergence-measure}. Moreover,
		$$
		\omega\nabla W\in W^{1,1}
		$$
		on a neighborhood of every compact $b$-band.
		
		Finally, let
		$$
		\Omega_{\alpha,\beta}:=\{\alpha<b<\beta\},
		$$
		where $\alpha$ and $\beta$ are regular values. This is a precompact
		smooth domain, whose outward unit normal is $\nu$ on $\Sigma_\beta$ and
		$-\nu$ on $\Sigma_\alpha$. Applying the Sobolev Gauss--Green theorem to
		the \(W^{1,1}\) vector field \(\omega\nabla W\) gives
		$$
		\int_{\Omega_{\alpha,\beta}}h_\omega\dd V_g
		=
		\int_{\Sigma_\beta}
		\omega\langle\nabla W,\nu\rangle\dd A_g
		-
		\int_{\Sigma_\alpha}
		\omega\langle\nabla W,\nu\rangle\dd A_g,
		$$
		which is \eqref{eq:weighted-Gauss-Green}.
	\end{proof}
	
	\begin{proposition}[Critical level Colding--Minicozzi identity]
		\label{prop:measure-identity}
		The regular-value functions $A$ and $B_1$ extend continuously to
		$(0,\infty)$, with $A\in C^1(0,\infty)$, $A\geq0$, and
		\begin{equation}\label{eq:A-first-variation}
			2(rA(r))'=B_1(r).
		\end{equation}
		Moreover,
		$$
		B_2-2\kappa\in L^1(0,R)
		\qquad\text{for every }R>0,
		$$
		and
		\begin{equation}\label{eq:exact-CM-with-measure}
			rB_1(r)
			=4rA(r)
			+\int_0^r\bigl(S_1(s)+B_2(s)-2\kappa(s)\bigr)\dd s
		\end{equation}
		for every $r>0$. If $R_g\geq0$ on $D$, then
		\begin{equation}\label{eq:one-sided-CM}
			rB_1(r)
			\geq4rA(r)
			+\int_0^r\bigl(B_2(s)-2\kappa(s)\bigr)\dd s.
		\end{equation}
	\end{proposition}
	
	\begin{proof}
		By Lemma~\ref{lem:CM-localization}, the regular value functions $A$ and
		$B_1$ admit continuous extensions to $(0,\infty)$, the function $A$ is
		locally absolutely continuous, and
		$$
		rA'(r)=\frac12B_1(r)-A(r)
		$$
		for almost every $r>0$. Hence
		$$
		2(rA(r))'=B_1(r)
		$$
		almost everywhere. Since $B_1$ is continuous, integration over compact
		intervals shows that $rA(r)$ is $C^1$ and proves
		\eqref{eq:A-first-variation} for every $r>0$. At each regular value,
		the divergence theorem and the normalization
		\eqref{eq:green-normalization} give
		$$
		r^{-2}\int_{\Sigma_r}|\nabla b|\dd A_g=4\pi,
		$$
		so $A(r)>0$. Since regular values are dense, continuity gives
		$A\geq0$ on $(0,\infty)$.
		
		Because $G=b^{-1}$ is harmonic away from $p$, one has
		\begin{equation}\label{eq:b-identities}
			\operatorname{div}_g(b^{-2}\nabla b)=0,
			\qquad
			\Delta_gb^2=6|\nabla b|^2
			\qquad\text{on }D\setminus\{p\}.
		\end{equation}
		Set
		$$
		f:=b^2,
		\qquad
		W:=|\nabla f|.
		$$
		On the regular set of $b$, the Bochner formula and the traced Gauss
		equation give the local scalar curvature level set identity. The resulting identity is the one used in \cite[\S2]{Stern}, \cite[Lemma~2.1]{MunteanuWangComparison}, and \cite[Proposition~1.17]{CM}. For
		$v=b^2$, it takes the form recorded in
		\cite[Proposition~1.17]{CM}:
		\begin{equation}\label{eq:pointwise-CM}
			\frac{\Delta_gW}{W}
			=\frac12R_g-K_{\Sigma_b}
			+\frac12\frac{|B|^2}{W^2}
			+\frac32\frac{\Hess(b^2)(\nu,\nu)}{b^2}.
		\end{equation}
		
		Apply Lemma~\ref{lem:positive-defect} with the weight $\omega=b^{-1}$.
		It follows that
		$$
		h_{b^{-1}}
		:=\operatorname{div}_g(b^{-1}\nabla W)
		\in L^1_{\mathrm{loc}}(D\setminus\{p\})
		$$
		in the weak sense, with no singular part supported on the critical set.
		On every regular level,
		\begin{equation}\label{eq:normal-W-Hess}
			\partial_\nu W=\Hess(b^2)(\nu,\nu).
		\end{equation}
		Consequently,
		$$
		\int_{\Sigma_s}b^{-1}\partial_\nu W\dd A_g
		=sB_1(s).
		$$
		
		On the regular set,
		$$
		h_{b^{-1}}
		=b^{-1}\Delta_gW
		-b^{-2}\langle\nabla b,\nabla W\rangle.
		$$
		Since $W=2b|\nabla b|$, equations
		\eqref{eq:pointwise-CM} and \eqref{eq:normal-W-Hess} imply that, for
		almost every regular value $s$,
		\begin{align*}
			\int_{\Sigma_s}\frac{h_{b^{-1}}}{|\nabla b|}\dd A_g
			&=S_1(s)+B_2(s)+2B_1(s)-2\kappa(s).
		\end{align*}
		
		By Sard's theorem, the critical values of $b$ form a null set; assign
		arbitrary values, say zero, to $B_2$ and $\kappa$ there. Lemma
		\ref{lem:positive-defect}, the vanishing volume of the critical set, and
		the coarea formula give, for regular values $0<\eps<r$,
		$$
		\int_\eps^r\int_{\Sigma_s}
		\frac{|h_{b^{-1}}|}{|\nabla b|}\dd A_g\dd s
		=
		\int_{\{\eps<b<r\}}|h_{b^{-1}}|\dd V_g
		<\infty.
		$$
		Applying the weighted Gauss--Green formula on $\{\eps<b<r\}$ therefore
		yields
		\begin{equation}\label{eq:annular-CM-identity}
			rB_1(r)-\eps B_1(\eps)
			=
			\int_\eps^r
			\bigl(S_1(s)+B_2(s)+2B_1(s)-2\kappa(s)\bigr)\dd s.
		\end{equation}
		
		It remains to control the pole. From the expansion
		\eqref{eq:robin-expansion}, inversion gives
		$$
		b=\varrho-m_D(p)\varrho^2+\theta,
		\qquad
		|\nabla^j\theta|=O(\varrho^{3-j}),
		\quad j=0,1,2.
		$$
		Together with
		$$
		\Hess(\varrho^2)=2g+O(\varrho^2)
		$$
		in geodesic normal coordinates, this yields
		$$
		b=\varrho+O(\varrho^2),
		\qquad
		|\nabla b|^2=1+O(\varrho),
		\qquad
		\Hess(b^2)=2g+O(\varrho).
		$$
		Thus $\nabla b\neq0$ near $p$, and every sufficiently small level is a
		sphere. The standard normal-coordinate area expansion then gives
		$$
		\begin{aligned}
			A(s)&=4\pi+O(s),&
			B_1(s)&=8\pi+O(s),&
			\kappa(s)&=4\pi,\\
			B_2(s)&=O(s^2),&
			S_1(s)&=O(s^2).
		\end{aligned}
		$$
		In particular,
		$$
		\eps A(\eps)\longrightarrow0,
		\qquad
		\eps B_1(\eps)\longrightarrow0,
		$$
		and both
		$$
		S_1+B_2+2B_1-2\kappa
		\quad\text{and}\quad
		S_1+B_2-2\kappa
		$$
		are integrable at the origin. Together with
		\eqref{eq:annular-CM-identity}, this also proves their local
		integrability away from the origin.
		
		Integrating \eqref{eq:A-first-variation} from $\eps$ to $r$ and letting
		$\eps\downarrow0$ gives
		$$
		\int_0^rB_1(s)\dd s=2rA(r).
		$$
		Letting $\eps\downarrow0$ in \eqref{eq:annular-CM-identity} therefore
		gives, for every regular value $r$,
		\begin{align*}
			rB_1(r)
			&=\int_0^r
			\bigl(S_1+B_2+2B_1-2\kappa\bigr)(s)\dd s\\
			&=4rA(r)
			+\int_0^r
			\bigl(S_1+B_2-2\kappa\bigr)(s)\dd s.
		\end{align*}
		The left side is continuous in $r$, while the right side is
		continuous by integrability of its level integrand. Since regular
		values are dense, the identity extends to every $r>0$, proving
		\eqref{eq:exact-CM-with-measure}.
		
		Finally, $S_1\in L^1(0,R)$ for every $R>0$ by
		Lemma~\ref{lem:CM-localization} and the pole estimates above. Hence
		$$
		B_2-2\kappa
		=
		\bigl(S_1+B_2-2\kappa\bigr)-S_1
		\in L^1(0,R).
		$$
		If $R_g\geq0$, then $S_1\geq0$ almost everywhere. Dropping its integral
		from \eqref{eq:exact-CM-with-measure} gives
		\eqref{eq:one-sided-CM}.
	\end{proof}

	\subsubsection{The Riccati inequality and quadratic growth}
	
	The argument below is adapted from the growth analysis in
	\cite[\S\S2.1--2.2]{CM}. On any interval on which $A>0$, define
	\begin{equation*}
		a(r):=\frac{rA'(r)}{A(r)}.
	\end{equation*}
	Since $A\in C^1$, the function $a$ is continuous on that interval.
	
	\begin{proposition}[Integral Riccati inequality]\label{prop:riccati}
		Assume that $R_g\geq0$ on $D$, that $A(r)>0$ for every $r\in[r_0,\infty)$, and that every regular level $\Sigma_r$ is connected and has genus at least one for $r\in[r_0,\infty)$. Then, for arbitrary $r_0\leq r_1<r_2$,
		\begin{equation}\label{eq:riccati}
			a(r_2)-a(r_1)
			\geq\int_{r_1}^{r_2}\left(1-\frac{a(r)^2}{4}\right)\frac{\dd r}{r}.
		\end{equation}
	\end{proposition}
	
	\begin{proof}
		On a compact band contained in $\{b\geq r_0\}$, define
		\begin{equation*}
			V:=\lambda(b)\nabla W,
			\qquad
			\lambda(r):=\frac{r^{-2}}{A(r)},
			\qquad
			W=|\nabla b^2|.
		\end{equation*}
		Since $A$ is positive and $C^1$, the weight $\lambda$ is positive and $C^1$. Lemma~\ref{lem:positive-defect} shows that $\operatorname{div}_gV\in L^1_{\mathrm{loc}}$ weakly and gives the exact Gauss--Green formula on regular level bands.
		
		At a regular value $r$, \eqref{eq:normal-W-Hess} gives
		\begin{align}
			\int_{\Sigma_r}\langle V,\nu\rangle\dd A_g
			&=\frac{r^{-2}}{A(r)}
			\int_{\Sigma_r}\Hess(b^2)(\nu,\nu)\dd A_g\notag\\
			&=\frac{B_1(r)}{A(r)}
			=2a(r)+2,
			\label{eq:V-flux}
		\end{align}
		where the last equality follows from \eqref{eq:A-first-variation}.
		
		On the regular set,
		$$
		\operatorname{div}_gV
		=\lambda(b)\Delta_gW+\lambda'(b)\langle\nabla b,\nabla W\rangle.
		$$
		At the level $b=r$, use $W=2r|\nabla b|$, \eqref{eq:pointwise-CM}, and \eqref{eq:normal-W-Hess}. The first term gives
		\begin{align}
			r\int_{\Sigma_r}\frac{\lambda(r)\Delta_gW}{|\nabla b|}\dd A_g
			&=\frac{2}{A(r)}\int_{\Sigma_r}\frac{\Delta_gW}{W}\dd A_g\notag\\
			&=\frac{S_1(r)+B_2(r)-2\kappa(r)}{A(r)}
			+3\frac{B_1(r)}{A(r)}.
			\label{eq:V-density-first-term}
		\end{align}
		For the weight derivative, the definition of $a$ gives
		\begin{equation*}
			\frac{\lambda'(r)}{\lambda(r)}
			=-\frac{2}{r}-\frac{A'(r)}{A(r)}
			=-\frac{2+a(r)}{r},
			\qquad
			r^3\lambda'(r)=-\frac{2+a(r)}{A(r)}.
		\end{equation*}
		Therefore
		\begin{align}
			r\int_{\Sigma_r}
			\frac{\lambda'(r)\langle\nabla b,\nabla W\rangle}{|\nabla b|}\dd A_g
			&=r\lambda'(r)\int_{\Sigma_r}\partial_\nu W\dd A_g\notag\\
			&=r^3\lambda'(r)B_1(r)
			=-(2+a(r))\frac{B_1(r)}{A(r)}.
			\label{eq:V-density-second-term}
		\end{align}
		Combining \eqref{eq:V-density-first-term}--\eqref{eq:V-density-second-term} with $B_1/A=2(a+1)$ yields the exact identity
		\begin{equation}\label{eq:V-density-exact}
			r\int_{\Sigma_r}
			\frac{\operatorname{div}_gV}{|\nabla b|}\dd A_g
			=\frac{S_1(r)+B_2(r)-2\kappa(r)}{A(r)}
			+2\bigl(1-a(r)^2\bigr)
		\end{equation}
		for almost every regular $r\geq r_0$.
		
		We also need the estimate relating $a$ and B
		2 Since \eqref{eq:b-identities} implies $\tr_gB=0$, every trace-free symmetric bilinear form in dimension three satisfies
		\begin{equation}\label{eq:tracefree-pointwise}
			|B|^2\geq\frac32 B(\nu,\nu)^2.
		\end{equation}
		Furthermore,
		\begin{align*}
			\int_{\Sigma_r}B(\nu,\nu)\dd A_g
			&=r^2\bigl(B_1(r)-2A(r)\bigr)
			=2r^2a(r)A(r).
		\end{align*}
		Cauchy--Schwarz, \eqref{eq:tracefree-pointwise}, and
		$$
		\int_{\Sigma_r}|\nabla b|^2\dd A_g=r^2A(r)
		$$
		give
		\begin{align*}
			B_2(r)
			&=\int_{\Sigma_r}\frac{|B|^2}{4r^2|\nabla b|^2}\dd A_g\\
			&\geq\frac{3}{8r^2}
			\int_{\Sigma_r}\frac{B(\nu,\nu)^2}{|\nabla b|^2}\dd A_g\\
			&\geq\frac{3}{8r^2}
			\frac{\left(\int_{\Sigma_r}B(\nu,\nu)\dd A_g\right)^2}
			{\int_{\Sigma_r}|\nabla b|^2\dd A_g}
			=\frac32a(r)^2A(r).
		\end{align*}
		Thus
		\begin{equation}\label{eq:a-B2}
			a(r)^2\leq\frac{2}{3}\frac{B_2(r)}{A(r)}.
		\end{equation}
		
		Because $R_g\geq0$, one has $S_1(r)\geq0$. Connectedness and genus at least one give $\kappa(r)\leq0$ by Gauss--Bonnet. Hence \eqref{eq:V-density-exact} and \eqref{eq:a-B2} imply
		\begin{equation}\label{eq:V-lower}
			r\int_{\Sigma_r}
			\frac{\operatorname{div}_gV}{|\nabla b|}\dd A_g
			\geq2-\frac12a(r)^2
		\end{equation}
		for almost every regular $r\geq r_0$.
		
		Let first $r_1<r_2$ be regular. Apply the Gauss--Green formula from Lemma~\ref{lem:positive-defect} to $V$ on $\{r_1<b<r_2\}$. Using \eqref{eq:V-flux}, coarea, and \eqref{eq:V-lower}, we obtain
		\begin{align*}
			2\bigl(a(r_2)-a(r_1)\bigr)
			&=\int_{r_1}^{r_2}\int_{\Sigma_r}
			\frac{\operatorname{div}_gV}{|\nabla b|}\dd A_g\dd r\\
			&\geq\int_{r_1}^{r_2}\left(2-\frac12a(r)^2\right)\frac{\dd r}{r}.
		\end{align*}
		Dividing by two proves \eqref{eq:riccati} for regular endpoints. The function $a$ is continuous, regular values are dense, and the right-hand integrand is continuous. Approximation therefore proves \eqref{eq:riccati} for arbitrary endpoints.
	\end{proof}
	
	We use the following comparison principle.
	
	\begin{lemma}\label{lem:integral-comparison}
		Let $F:\R\to\R$ be locally Lipschitz. Suppose that $\alpha:[0,T]\to\R$ is continuous and satisfies
		\begin{equation}\label{eq:integral-supersolution}
			\alpha(t_2)-\alpha(t_1)
			\geq\int_{t_1}^{t_2}F(\alpha(t))\dd t
			\qquad(0\leq t_1<t_2\leq T).
		\end{equation}
		Let $\psi$ solve $\psi'=F(\psi)$ with $\psi(0)\leq\alpha(0)$. Then $\psi(t)\leq\alpha(t)$ for every $t\in[0,T]$ for which $\psi$ exists.
	\end{lemma}
	
	\begin{proof}
		Suppose that the open set $\{t:\psi(t)>\alpha(t)\}$ is nonempty, and let $(s,t_*)$ be one of its connected components. By continuity,
		$$
		\psi(s)=\alpha(s),
		\qquad
		\psi>\alpha\quad\text{on }(s,t_*).
		$$
		Both functions have bounded range on the compact interval $[s,t_*]$, so $F$ is Lipschitz there with some constant $L$. Subtracting \eqref{eq:integral-supersolution} from the integral equation for $\psi$ gives, for $s<t<t_*$,
		$$
		0<\psi(t)-\alpha(t)
		\leq L\int_s^t\bigl(\psi(u)-\alpha(u)\bigr)\dd u.
		$$
		Gr\"onwall's inequality forces $\psi-\alpha\equiv0$ on $[s,t_*)$, a contradiction.
	\end{proof}
	
	\begin{proposition}[Uniform quadratic growth]\label{prop:uniform-quadratic}
		Let $\beta=b_1(M)$ and assume $R_g\geq0$. Suppose that, for a Dirichlet Green function as above, every regular level $\Sigma_r$ is connected and has genus at least two whenever $r\geq r_*$. Set
		\begin{equation*}
			R_0:=2(\beta+2)r_*,
			\qquad
			r_1:=2R_0=4(\beta+2)r_*,
		\end{equation*}
		and
		\begin{equation*}
			c_*:=\frac{2\pi}{r_1^2}\left(\frac34\right)^4
			=\frac{81\pi}{128r_1^2}>0.
		\end{equation*}
		Then
		\begin{equation}\label{eq:quadratic-growth}
			A(r)\geq c_*r^2
			\qquad\text{for all }r\geq r_1.
		\end{equation}
		The constants $r_1=r_1(\beta,r_*)$ and $c_*=c_*(\beta,r_*)$ depend only on $\beta$ and $r_*$, not on the outer Dirichlet domain.
	\end{proposition}
	
	\begin{proof}
		For a regular level with components of genera $g_1,\ldots,g_{N(r)}$, Gauss--Bonnet and \eqref{eq:level-component-bound} give
		\begin{equation}\label{eq:kappa-upper-small}
			\kappa(r)=4\pi\sum_{j=1}^{N(r)}(1-g_j)
			\leq4\pi N(r)
			\leq4\pi(\beta+1).
		\end{equation}
		For $r\geq r_*$, connectedness and genus at least two give
		\begin{equation}\label{eq:kappa-gap}
			\kappa(r)\leq-4\pi.
		\end{equation}
		These estimates hold for almost every parameter, which is sufficient under our critical-value convention. Moreover, $B_2-2\kappa$ is locally integrable by Proposition~\ref{prop:measure-identity}; the lower bounds above control the negative part of $-2\kappa$. Because $R_g\geq0$ and $B_2\geq0$, inserting \eqref{eq:kappa-upper-small}--\eqref{eq:kappa-gap} into \eqref{eq:one-sided-CM} gives, for $r\geq r_*$,
		\begin{align}
			rB_1(r)
			&\geq4rA(r)-8\pi(\beta+1)r_*+8\pi(r-r_*)\notag\\
			&=4rA(r)+8\pi r-8\pi(\beta+2)r_*.
			\label{eq:genus-gap-B1}
		\end{align}
		Using $B_1=2(rA)'$ and dividing \eqref{eq:genus-gap-B1} by $2r$ yields
		\begin{equation*}
			rA'(r)
			\geq A(r)+4\pi-\frac{4\pi(\beta+2)r_*}{r}.
		\end{equation*}
		For $r\geq R_0$, the last term is at most $2\pi$, so
		\begin{equation}\label{eq:A-plus-2pi}
			rA'(r)\geq A(r)+2\pi.
		\end{equation}
		Equivalently,
		$$
		\left(\frac{A(r)}{r}\right)'\geq\frac{2\pi}{r^2}.
		$$
		Integrating from $R_0$ to $2R_0=r_1$ and using $A\geq0$, we obtain
		\begin{equation}\label{eq:A-r1}
			A(r_1)\geq2\pi.
		\end{equation}
		In addition, integrating \eqref{eq:A-plus-2pi} from $R_0$ to any $r>R_0$ shows that $A(r)>0$; in particular, $A$ is positive on $[r_1,\infty)$. Equation \eqref{eq:A-plus-2pi} also gives
		\begin{equation}\label{eq:a-r1}
			a(r_1)=\frac{r_1A'(r_1)}{A(r_1)}\geq1.
		\end{equation}
		
		Since all regular levels beyond $r_1$ have genus at least two, Proposition~\ref{prop:riccati} applies. Set logarithmic time $t=\log(r/r_1)$ and
		$$
		\alpha(t):=a(r_1e^t).
		$$
		Then \eqref{eq:riccati} says precisely that $\alpha$ is an integral supersolution of
		$$
		\alpha'=1-\frac{\alpha^2}{4}.
		$$
		The solution with initial value $1$ is, in the original radial variable,
		\begin{equation*}
			\varphi(r):=2\frac{3r-r_1}{3r+r_1};
		\end{equation*}
		it satisfies
		$$
		r\varphi'(r)=1-\frac{\varphi(r)^2}{4},
		\qquad
		\varphi(r_1)=1.
		$$
		Since $a(r_1)\geq1$ by \eqref{eq:a-r1}, Lemma~\ref{lem:integral-comparison} gives
		\begin{equation}\label{eq:a-phi}
			a(r)\geq\varphi(r)
			\qquad(r\geq r_1).
		\end{equation}
		Integrating \eqref{eq:a-phi} gives
		$$
		\log\frac{A(r)}{A(r_1)}
		\geq\int_{r_1}^r\varphi(s)\frac{\dd s}{s}
		=2\log\frac{r}{r_1}-4\log\frac{4r}{3r+r_1}.
		$$
		Therefore
		\begin{equation}\label{eq:A-explicit-lower}
			A(r)
			\geq A(r_1)\left(\frac{r}{r_1}\right)^2
			\left(\frac{3r+r_1}{4r}\right)^4.
		\end{equation}
		Since $(3r+r_1)/(4r)\geq3/4$ and \eqref{eq:A-r1} holds, \eqref{eq:A-explicit-lower} implies \eqref{eq:quadratic-growth}.
	\end{proof}

	\subsubsection{The boundary limit of the level set quantity}
	
	\begin{proposition}\label{prop:boundary-flux}
		Let $D\Subset M$ be smooth and let $G=G_D(p,\cdot)$. Set
		$$
		q:=-\partial_\nu G>0
		\qquad\text{on }\partial D.
		$$
		Then
		\begin{equation}\label{eq:boundary-flux-limit}
			\lim_{r\to\infty}\frac{A(r)}{r^2}
			=\int_{\partial D}q^2\dd A_g.
		\end{equation}
	\end{proposition}
	
	\begin{proof}
		On $\{b=r\}=\{G=r^{-1}\}$,
		$$
		\nabla b=-G^{-2}\nabla G=-r^2\nabla G.
		$$
		Consequently,
		\begin{equation}\label{eq:A-G-level}
			\frac{A(r)}{r^2}
			=\int_{\{G=r^{-1}\}}|\nabla G|^2\dd A_g.
		\end{equation}
		
		In a sufficiently small collar of $\partial D$, use inward boundary
		normal coordinates $(y,s)$, where $y\in\partial D$ and
		$s=d_g(\,\cdot\,,\partial D)$. Boundary regularity and the Hopf lemma give, uniformly on the compact boundary,
		\begin{align}
			G(y,s)&=q(y)s+O(s^2),\notag\\
			\partial_sG(y,s)&=q(y)+O(s),
			\qquad
			\nabla_{\partial D}G(y,s)=s\nabla_{\partial D}q(y)+O(s^2).
			\label{eq:grad-G-Fermi}
		\end{align}
		Choose the collar so that the compact set outside it has a positive minimum of $G$. For every sufficiently small $\eps>0$, the entire level $\{G=\eps\}$ therefore lies in the collar. Since $\min_{\partial D}q>0$, the implicit function theorem shows that this level is a normal graph
		\begin{equation}\label{eq:level-graph}
			s=s_\eps(y)=\frac{\eps}{q(y)}+O(\eps^2).
		\end{equation}
		Equations \eqref{eq:grad-G-Fermi}--\eqref{eq:level-graph} imply
		$$
		|\nabla G|^2\big|_{\{G=\eps\}}=q(y)^2+O(\eps),
		\qquad
		\dd A_\eps=(1+O(\eps))\dd A_{\partial D}.
		$$
		It follows that
		$$
		\lim_{\eps\downarrow0}\int_{\{G=\eps\}}|\nabla G|^2\dd A_g
		=\int_{\partial D}q^2\dd A_g.
		$$
		Taking $\eps=r^{-1}$ in \eqref{eq:A-G-level} proves \eqref{eq:boundary-flux-limit}.
	\end{proof}
	
	Proposition~\ref{prop:uniform-quadratic} supplies a domain-independent lower bound for $A(r)/r^2$, and Proposition~\ref{prop:boundary-flux} identifies its boundary limit. Their combination is the uniform estimate required by the Hadamard formula.
	
	\subsection{Proof of the low-genus separator theorem}\label{sec:low-genus-proof}

	\begin{proof}[Proof of Theorem~\ref{thm:low-genus}]
		Suppose the conclusion is false. Then there exists a compact set $K_0\Subset M$ such that every admissible separator $U$ containing $K_0$ satisfies
		\begin{equation}\label{eq:genus-core-assumption}
			\genus(\partial U)\geq2.
		\end{equation}
		Let $S$ be the compact set in Lemma~\ref{lem:compact-detector}. Enlarge $K_0$ to a compact set $K$ containing $S$ and $p$; property \eqref{eq:genus-core-assumption} is preserved.
		
		Choose the bounded gradient proper Morse function $\rho$ from Proposition~\ref{prop:morse-function} and its filled domains $D_\tau$. By \eqref{eq:filled-exhaustion}, there is a regular value $\tau_0$ such that
		\begin{equation*}
			K\Subset D_{\tau_0}.
		\end{equation*}
		Let $G_\tau=G_{D_\tau}(p,\cdot)$ for regular $\tau\geq\tau_0$, with the convention $G_\tau(p)=+\infty$. Choose $\delta>0$ with $\overline{B_g(p,2\delta)}\subset D_{\tau_0}$. By the pole expansion, after decreasing $\delta$ there is $c_0>0$ such that
		$$
		G_{\tau_0}\geq c_0
		\qquad\text{on }B_g(p,2\delta)\setminus\{p\}.
		$$
		If $K\setminus B_g(p,\delta)$ is nonempty, it is a compact subset of $D_{\tau_0}\setminus\{p\}$, and we set
		$$
		c_1:=\min_{K\setminus B_g(p,\delta)}G_{\tau_0}>0.
		$$
		If it is empty, set $c_1:=c_0$. Choose $r_*>0$ so that $r_*^{-1}<\min\{c_0,c_1\}$. Then
		\begin{equation*}
			G_{\tau_0}>r_*^{-1}
			\qquad\text{on }K\setminus\{p\}.
		\end{equation*}
		Domain monotonicity and $D_{\tau_0}\subset D_\tau$ give
		\begin{equation*}
			K\subset U_{\tau,r}:=\{p\}\cup\{G_\tau>r^{-1}\}
		\end{equation*}
		for every regular $\tau\geq\tau_0$ and every $r\geq r_*$.
		
		Fix such a parameter $\tau$. By Lemma~\ref{lem:green-connectivity}, $U_{\tau,r}$ and its complement are connected for every regular $r$. Because $S\subset K\subset U_{\tau,r}$ when $r\geq r_*$, Corollary~\ref{cor:large-level-connected} gives that $\partial U_{\tau,r}$ is connected. It is therefore an admissible separator containing $K_0$, and \eqref{eq:genus-core-assumption} yields
		\begin{equation*}
			\genus(\partial U_{\tau,r})\geq2
			\qquad(r\geq r_*,\ r\text{ regular}).
		\end{equation*}
		For smaller regular levels, Lemmas~\ref{lem:green-connectivity} and~\ref{lem:boundary-components} give the uniform component bound
		\begin{equation*}
			\#\pi_0(\partial U_{\tau,r})\leq b_1(M)+1.
		\end{equation*}
		
		The constants $r_*$ and $\beta=b_1(M)$ are independent of $\tau$. Applying Proposition~\ref{prop:uniform-quadratic} to $G_\tau$ therefore produces a constant $c_*=c_*(\beta,r_*)>0$, independent of $\tau$, such that
		\begin{equation*}
			A_\tau(r)\geq c_*r^2
			\qquad\text{for all }r\geq r_1.
		\end{equation*}
		By Proposition~\ref{prop:boundary-flux},
		\begin{equation}\label{eq:q-tau-lower}
			\int_{\partial D_\tau}q_\tau^2\dd A_g
			=\lim_{r\to\infty}\frac{A_\tau(r)}{r^2}
			\geq c_*.
		\end{equation}
		
		Let $a<b$ be regular parameters with $\tau_0\leq a<b$. Inserting \eqref{eq:q-tau-lower} into \eqref{eq:hadamard-lower-integrated} gives
		\begin{equation}\label{eq:robin-linear-growth}
			m_{D_b}(p)-m_{D_a}(p)
			\geq\frac{c_*}{4\pi L}(b-a).
		\end{equation}
		Critical values between $a$ and $b$ contribute only the nonnegative jumps already incorporated in Proposition~\ref{prop:critical-jumps}. Letting $b\to\infty$ through regular values in \eqref{eq:robin-linear-growth} forces
		$$
		m_{D_b}(p)\longrightarrow+\infty.
		$$
		This contradicts the global upper bound
		$$
		m_{D_b}(p)\leq m_M(p)<\infty
		$$
		from \eqref{eq:global-robin-bound}. Hence no compact set $K_0$ satisfying \eqref{eq:genus-core-assumption} exists. Equivalently, every compact subset of $M$ is contained in an admissible separator of genus at most one.
	\end{proof}

	\begin{proof}[Proof of Corollary~\ref{cor:nonparabolic-contractible}]
		A contractible manifold is orientable, has $b_1=0$, and has one end; see \cite[Lemmas 2.1--2.2]{YanZhu}. Theorem~\ref{thm:low-genus} therefore gives property $\mathrm{(LG)}$, and Proposition~\ref{prop:topological-reduction}\textup{(i)} yields $M\cong_{\mathrm{diff}}\R^3$.
	\end{proof}
	
	\begin{proof}[Proof of Corollary~\ref{cor:nonparabolic-handlebody}]
		The manifold $M_\gamma$ is orientable, one-ended, and has $b_1(M_\gamma)=\gamma<\infty$. Theorem~\ref{thm:low-genus} gives property $\mathrm{(LG)}$, and Proposition~\ref{prop:topological-reduction}\textup{(ii)} gives $\gamma\leq1$.
	\end{proof}

	\section{Metric replacement}\label{sec:metric-replacement}
	
	We will prove prove Theorem~\ref{thm:metric-replacement} in this section. The construction is deliberately separated into two parts. Kazdan's deformation creates strict scalar curvature positivity while preserving completeness by uniform equivalence. The Evans potential deformation then creates nonparabolicity while retaining strict positivity.
	
	\subsection{Deformation from nonnegative to positive scalar curvature}
	
	We first record the complete manifold deformation theorem that will be used to create strict scalar curvature positivity. The uniform equivalence in the conclusion is important because it preserves completeness.
	
	\begin{lemma}\label{lem:kazdan-supersolution}
		Let $(M^n,h)$ be complete, with $n\geq3$, and set
		$$
		c_n:=\frac{4(n-1)}{n-2},
		\qquad
		\mathcal L_h:=-c_n\Delta_h+R_h.
		$$
		Suppose that there is a smooth bounded domain $\Omega\Subset M$ such that
		$$
		\spt(R_h^-)\Subset\Omega,
		\qquad
		R_h^-:=\max\{-R_h,0\},
		$$
		and the principal Neumann eigenvalue of $\mathcal L_h$ on $\Omega$ is positive. Then there are a smooth function $w$ and constants $0<c_0\leq C_0<\infty$ such that
		\begin{equation*}
			c_0\leq w\leq C_0,
			\qquad
			\mathcal L_hw>0
			\quad\text{on }M.
		\end{equation*}
	\end{lemma}
	
	\begin{proof}
		Apply Kazdan's Theorem \cite[Theorem A]{Kazdan} to the Schr\"odinger operator
		$$
		-\Delta_h+c_n^{-1}R_h.
		$$
		The negative part of its potential is supported in $\Omega$, and positivity of its principal Neumann eigenvalue is equivalent to positivity of the principal Neumann eigenvalue of $\mathcal L_h$. Kazdan's conclusion gives a positive solution bounded uniformly above and below and satisfying the strict supersolution inequality. This operator formulation is also recorded explicitly in \cite[Proposition A.1]{LeeLesourdUnger}.
	\end{proof}
	
	The following proposition is standard.
	\begin{proposition}[Kazdan deformation]\label{prop:kazdan-deformation}
		Let $(M^n,g)$ be complete, with $n\geq3$, and suppose that
		$$
		R_g\geq0,
		\qquad
		\Ric_g\not\equiv0.
		$$
		Then there is a complete Riemannian metric $g_+$ on $M$ such that
		$
		R_{g_+}>0
		$
		. Moreover, $g_+$ is uniformly equivalent to $g$: there are constants $0<c\leq C<\infty$ such that
		$$
		cg\leq g_+\leq Cg.
		$$
	\end{proposition}
	
	\begin{proof}
		Set
		$$
		c_n:=\frac{4(n-1)}{n-2},
		\qquad
		\mathcal L_h:=-c_n\Delta_h+R_h.
		$$
		For a smooth connected precompact domain $\Omega\Subset M$, let $\mu_1(\Omega,h)$ be the first Neumann eigenvalue of $\mathcal L_h$ on $\Omega$:
		\begin{equation}\label{eq:neumann-conformal-eigenvalue}
			\mu_1(\Omega,h)
			=\inf_{0\neq v\in H^1(\Omega)}
			\frac{\displaystyle\int_\Omega
				\bigl(c_n|\nabla v|_h^2+R_hv^2\bigr)\dd V_h}
			{\displaystyle\int_\Omega v^2\dd V_h}.
		\end{equation}
		
		Suppose first that $R_g$ is not identically zero. Since $R_g\geq0$, one can choose $\Omega$ so that $R_g>0$ somewhere in $\Omega$. Formula \eqref{eq:neumann-conformal-eigenvalue} gives $\mu_1(\Omega,g)\geq0$. Equality is impossible: a first eigenfunction for the zero eigenvalue would make both nonnegative terms in the numerator vanish, so it would be a nonzero constant and would force $R_g\equiv0$ on $\Omega$. Hence
		$$
		\mu_1(\Omega,g)>0.
		$$
		
		It remains to consider the scalar flat case $R_g\equiv0$. Choose a point at which $\Ric_g\neq0$, a smooth precompact domain $\Omega$ containing that point, and a nonnegative cutoff $\eta\in C_c^\infty(\Omega)$ such that
		$$
		\int_\Omega\eta|\Ric_g|^2\dd V_g>0.
		$$
		For small $t\geq0$, define
		$$
		g_t:=g-t\eta\Ric_g.
		$$
		The perturbation is supported in $\Omega$, so $g_t=g$ near $\partial\Omega$ and outside a compact set, and $g_t$ is uniformly equivalent to $g$ for small $t$. Since $\Omega$ is connected, the zero Neumann eigenvalue at $t=0$ is simple, with normalized eigenfunction $v_0=\Vol_g(\Omega)^{-1/2}$. Because $g_t=g$ near $\partial\Omega$, the Neumann realizations are associated with a selfadjoint holomorphic family of closed forms of type \textup{(a)} on the common domain $H^1(\Omega)$. By \cite[Chapter VII, Theorem 4.2 and Remark 4.22]{Kato}, the simple eigenvalue and a normalized eigenfunction depend holomorphically on $t$; the first variation formula used below is \cite[Chapter VII, equation (4.56)]{Kato}. Writing the numerator and denominator of \eqref{eq:neumann-conformal-eigenvalue} as $Q_t$ and $N_t$, respectively, the normalized eigenfunction $v_0=\Vol_g(\Omega)^{-1/2}$ satisfies
		$$
		\mu_1'(0)=Q_0'(v_0)-\mu_1(0)N_0'(v_0)=Q_0'(v_0).
		$$
		Since $R_g=0$ and $\nabla v_0=0$, the metric variation of the gradient term vanishes, while the volume-form variation is multiplied by $R_g$. Thus
		$$
		\left.\frac{d}{dt}\right|_{t=0}\mu_1(\Omega,g_t)
		=\frac{1}{\Vol_g(\Omega)}\int_\Omega DR_g(h)\dd V_g,
		\qquad h=-\eta\Ric_g.
		$$
		The linearization
		$$
		DR_g(h)=-\Delta_g(\tr_g h)+\operatorname{div}_g\operatorname{div}_g h
		-\langle\Ric_g,h\rangle_g
		$$
		and $\spt h\Subset\Omega$ make the divergence terms integrate to zero. Hence
		\begin{equation*}
			\left.\frac{d}{dt}\right|_{t=0}\mu_1(\Omega,g_t)
			=\frac{1}{\Vol_g(\Omega)}
			\int_\Omega\eta|\Ric_g|^2\dd V_g>0.
		\end{equation*}
		Thus $\mu_1(\Omega,g_{t_0})>0$ for some sufficiently small $t_0>0$.
		
		In either case we have obtained a complete metric $g_0$, uniformly equivalent to $g$, and a smooth precompact domain $\Omega$ such that
		$$
		\mu_1(\Omega,g_0)>0.
		$$
		In the first case $g_0=g$ and $R_{g_0}\geq0$ on all of $M$. In the scalar-flat case, the cutoff satisfies $\spt\eta\Subset\Omega$, and $g_{t_0}=g$ on a neighborhood of $M\setminus\spt\eta$. Since scalar curvature is a local expression in the metric and its first two derivatives,
		$$
		R_{g_0}=R_g=0
		\qquad\text{on }M\setminus\spt\eta.
		$$
		Thus in either case $\spt(R_{g_0}^{-})\Subset\Omega$. Lemma~\ref{lem:kazdan-supersolution} therefore supplies a smooth function $w>0$ and constants $c_0,C_0>0$ satisfying
		\begin{equation}\label{eq:kazdan-bounded-supersolution}
			c_0\leq w\leq C_0,
			\qquad
			\mathcal L_{g_0}w>0
			\quad\text{on }M.
		\end{equation}
		Define
		$$
		g_+:=w^{\frac{4}{n-2}}g_0.
		$$
		The conformal scalar curvature formula gives
		$$
		R_{g_+}
		=w^{-\frac{n+2}{n-2}}\mathcal L_{g_0}w>0.
		$$
		The two-sided bounds in \eqref{eq:kazdan-bounded-supersolution}, together with the uniform equivalence of $g_0$ and $g$, show that $g_+$ is uniformly equivalent to $g$. In particular, $g_+$ is complete.
	\end{proof}
	
	\subsection{An Evans potential conformal deformation}
	
	We use the following consequence of the Evans potential theorem.
	
	\begin{lemma}\label{lem:exterior-Evans}
		Let $(M,g)$ be a complete, one-ended, noncompact parabolic Riemannian manifold. Then there exist a connected exterior domain $E\subset M$ with smooth compact boundary and a function
		$$
		u\in C^\infty(E)\cap C^\infty(\overline E)
		$$
		such that
		\begin{equation}\label{eq:exterior-Evans-properties}
			\begin{aligned}
				&\Delta_gu=0\quad\text{on }E,
				\qquad
				u=0\quad\text{on }\partial E,
				\qquad
				u>0\quad\text{on }E,\\
				&u(x)\longrightarrow+\infty\quad\text{as }x\to\infty.
			\end{aligned}
		\end{equation}
		In particular, $u:\overline E\to[0,\infty)$ is proper.
	\end{lemma}
	
	\begin{proof}
		Choose a smooth compact domain $K_0\Subset M$ with nonempty interior; such a set is nonpolar. By the theorem of Hansen--Netuka \cite[Corollary~1.2]{HansenNetuka}, there is a positive harmonic function
		$$
		h:M\setminus K_0\longrightarrow(0,\infty)
		$$
		that tends to $+\infty$ at infinity. Choose a smooth compact domain $K_1$ with $K_0\Subset\Int K_1$. Since $h$ is smooth near $\partial K_1$, Sard's theorem allows us to choose a regular value
		$$
		c>\max_{\partial K_1}h.
		$$
		Because $h\to+\infty$ at infinity, first choose a smooth compact domain containing $K_1$ outside which $h>c$. Fill all of its precompact complementary components. One-endedness then gives a smooth compact domain $K_2$ with $K_1\Subset\Int K_2$, connected complement, and
		$$
		h>c\quad\text{on }M\setminus K_2.
		$$
		Let $E$ be the connected component of
		$$
		\{x\in M\setminus K_1:h(x)>c\}
		$$
		that contains $M\setminus K_2$. Then $E$ is connected and $M\setminus E\subset K_2$, so $E$ is an exterior domain. The inequality $c>\max_{\partial K_1}h$ shows that $\partial E$ does not meet $\partial K_1$; hence
		$$
		\partial E\subset\{h=c\}.
		$$
		Since $c$ is a regular value, $\partial E$ is a smooth compact hypersurface. Setting
		$$
		u:=h-c\quad\text{on }E
		$$
		gives all assertions in \eqref{eq:exterior-Evans-properties}. Properness follows from $h(x)\to+\infty$ at infinity and compactness of $M\setminus E$.
	\end{proof}
	
	We now prove the metric bridge stated in the introduction.
	
	\begin{proof}[Proof of Proposition~\ref{prop:nonparabolicization}]
		Let $g$ be a complete metric with $R_g>0$. If $(M,g)$ is already nonparabolic, take $\widehat g=g$. We therefore assume that $(M,g)$ is parabolic. Lemma~\ref{lem:exterior-Evans} gives a smooth exterior domain $E$ and a proper harmonic function
		\begin{equation*}
			u:E\longrightarrow[0,\infty)
		\end{equation*}
		satisfying \eqref{eq:exterior-Evans-properties}
		
		Choose regular values $0<a<b$. Let
		$$
		\eta:[0,\infty)\longrightarrow[0,1]
		$$
		be smooth and nondecreasing, with
		$$
		\eta=0\quad\text{on }[0,a],
		\qquad
		\eta=1\quad\text{on }[b,\infty).
		$$
		For $\varepsilon>0$, set
		\begin{equation*}
			F_\varepsilon(t):=1+\varepsilon\int_0^t\eta(s)\dd s.
		\end{equation*}
		Then
		\begin{equation*}
			F_\varepsilon\geq1,
			\qquad
			F_\varepsilon'=\varepsilon\eta,
			\qquad
			F_\varepsilon''=\varepsilon\eta',
		\end{equation*}
		and $F_\varepsilon$ is affine with slope $\varepsilon$ on $[b,\infty)$. Define a global smooth positive function by
		\begin{equation*}
			\phi:=
			\begin{cases}
				F_\varepsilon(u),&\text{on }E,\\
				1,&\text{on }M\setminus E.
			\end{cases}
		\end{equation*}
		This is smooth because $F_\varepsilon(u)=1$ throughout the collar where $u\leq a$. Put
		\begin{equation*}
			\widehat g:=\phi^4g.
		\end{equation*}
		
		We first choose $\varepsilon$ so that $R_{\widehat g}>0$. The transition band
		$$
		\mathcal K_{a,b}:=\{x\in E:a\leq u(x)\leq b\}
		$$
		is compact. Hence
		$$
		R_*:=\min_{\mathcal K_{a,b}}R_g>0,
		\qquad
		Q_*:=\max_{\mathcal K_{a,b}}
		\eta'(u)|\nabla u|_g^2<\infty.
		$$
		Choose $\varepsilon>0$ so small that
		\begin{equation}\label{eq:epsilon-choice}
			8\varepsilon Q_*<R_*.
		\end{equation}
		In dimension three, the conformal scalar curvature formula is
		\begin{equation*}
			R_{\widehat g}
			=\phi^{-5}\bigl(-8\Delta_g\phi+R_g\phi\bigr).
		\end{equation*}
		Since $u$ is harmonic,
		\begin{equation}\label{eq:laplacian-phi}
			\Delta_g\phi
			=F_\varepsilon''(u)|\nabla u|_g^2
			=\varepsilon\eta'(u)|\nabla u|_g^2.
		\end{equation}
		On $\mathcal K_{a,b}$, equations \eqref{eq:epsilon-choice}--\eqref{eq:laplacian-phi} and $\phi\geq1$ give
		$$
		-8\Delta_g\phi+R_g\phi
		\geq R_*-8\varepsilon Q_*>0.
		$$
		Outside $\mathcal K_{a,b}$ one has $F_\varepsilon''(u)=0$, and therefore
		$$
		-8\Delta_g\phi+R_g\phi=R_g\phi>0.
		$$
		Thus $R_{\widehat g}>0$ everywhere. Moreover, $\phi\geq1$, so every divergent curve has at least its original $g$-length. The completeness of $g$ therefore implies the completeness of $\widehat g$.
		
		It remains to prove that $\widehat g$ is nonparabolic. Define the compact set and its exhaustion by
		\begin{align*}
			K_b&:=(M\setminus E)\cup\{x\in E:u(x)\leq b\},\\
			\Omega_T&:=(M\setminus E)\cup\{x\in E:u(x)<T\},
		\end{align*}
		where $T>b$ is a regular value. Properness of $u$ makes $\Omega_T$ precompact, and regular values tending to infinity give a smooth exhaustion of $M$. We use the capacity normalization
		$$
		\operatorname{Cap}_{\widehat g}(K;\Omega)
		:=\inf\left\{\int_{\Omega\setminus K}|\nabla v|_{\widehat g}^2\dd V_{\widehat g}:
		v\in H^1(\Omega\setminus K),\ v|_{\partial K}=1,\ v|_{\partial\Omega}=0\right\}.
		$$
		Set
		\begin{equation*}
			I_T:=\int_b^T F_\varepsilon(s)^{-2}\dd s
		\end{equation*}
		and, on the annulus $\{b<u<T\}$, define
		\begin{equation}\label{eq:capacity-potential}
			h_T(x):=
			\frac{\displaystyle\int_{u(x)}^T F_\varepsilon(s)^{-2}\dd s}
			{\displaystyle I_T},
		\end{equation}
		extending $h_T$ by $1$ on $K_b$. Then $h_T=0$ on $\partial\Omega_T=\{u=T\}$.
		
		For any smooth one-variable function $H$, the conformal Laplacian formula for functions gives
		\begin{equation*}
			\Delta_{\widehat g}H(u)
			=F_\varepsilon(u)^{-4}|\nabla u|_g^2
			\left(H''(u)+2\frac{F_\varepsilon'(u)}{F_\varepsilon(u)}H'(u)\right).
		\end{equation*}
		The one-variable function in \eqref{eq:capacity-potential} satisfies
		$$
		H_T'(t)=-\frac{F_\varepsilon(t)^{-2}}{I_T},
		\qquad
		H_T''(t)+2\frac{F_\varepsilon'(t)}{F_\varepsilon(t)}H_T'(t)=0.
		$$
		Thus $h_T$ is $\widehat g$-harmonic on $\{b<u<T\}$. Every component of this annulus meets both boundary levels: otherwise $u$ would be harmonic on a precompact component with the single constant boundary value $b$ or $T$, contradicting the maximum principle. Hence the Dirichlet principle applies componentwise, and $h_T$ is the relative capacity potential of $K_b$ in $\Omega_T$.
		
		For every regular value $t>0$, let
		\begin{equation}\label{eq:Evans-flux}
			J(t):=\int_{\{u=t\}}|\nabla u|_g\dd A_g.
		\end{equation}
		For regular values $0<s<t$, the divergence theorem on the compact band $\{s<u<t\}$, summing over all components, gives $J(s)=J(t)$. Denote the common positive value by $J$. The conformal scaling laws in dimension three give
		\begin{equation*}
			|\nabla h_T|_{\widehat g}^2\dd V_{\widehat g}
			=F_\varepsilon(u)^2|\nabla h_T|_g^2\dd V_g.
		\end{equation*}
		Using coarea, \eqref{eq:capacity-potential}, and \eqref{eq:Evans-flux}, we obtain 
		\begin{align}
			\operatorname{Cap}_{\widehat g}(K_b;\Omega_T)
			&=\int_{\{b<u<T\}}|\nabla h_T|_{\widehat g}^2\dd V_{\widehat g}\notag\\
			&=J\int_b^T F_\varepsilon(t)^2H_T'(t)^2\dd t
			=\frac{J}{I_T}.
			\label{eq:relative-capacity-formula}
		\end{align}
		For $t\geq b$, the function $F_\varepsilon$ is exactly
		\begin{equation*}
			F_\varepsilon(t)=F_\varepsilon(b)+\varepsilon(t-b).
		\end{equation*}
		Consequently,
		\begin{equation*}
			I_\infty:=\int_b^\infty F_\varepsilon(t)^{-2}\dd t
			=\frac{1}{\varepsilon F_\varepsilon(b)}<\infty.
		\end{equation*}
		Relative capacities decrease to the global capacity along a smooth exhaustion; see \cite[Section~2]{HurtadoPalmerRitore}. Letting $T\to\infty$ through regular values in \eqref{eq:relative-capacity-formula} yields
		\begin{equation*}
			\operatorname{Cap}_{\widehat g}(K_b)
			=\frac{J}{I_\infty}
			=J\varepsilon F_\varepsilon(b)>0.
		\end{equation*}
		A complete manifold is nonparabolic exactly when some compact set has
		positive capacity; see, for example,
		\cite[Remark~3.4]{HurtadoPalmerRitore}. Hence $(M,\widehat g)$ is nonparabolic.
	\end{proof}
	
	\subsection{The combined metric replacement theorem}
	
	\begin{proof}[Proof of Theorem~\ref{thm:metric-replacement}]
		Proposition~\ref{prop:kazdan-deformation} gives a complete metric $g_+$ on $M$, uniformly equivalent to $g$, such that $R_{g_+}>0$. Since $M$ is one-ended, Proposition~\ref{prop:nonparabolicization}, applied to $g_+$, gives a complete nonparabolic metric $\widehat g$ with $R_{\widehat g}>0$.
	\end{proof}

	\section{Proof of the main theorems}\label{sec:main-theorems}
	
	\subsection{The flat branch}
	
	The following elementary packing lemma is sufficient for the open handlebody case.
	
	\begin{lemma}\label{lem:flat-group-growth}
		Let $(N^n,h)$ be a complete flat Riemannian manifold with finitely generated fundamental group $\Gamma$. Then $\Gamma$ has polynomial growth of degree at most $n$.
	\end{lemma}
	
	\begin{proof}
		The universal cover of $(N,h)$ is a complete simply connected flat manifold and hence isometric to Euclidean space $\R^n$. After making this identification, $\Gamma$ acts freely and properly discontinuously on $\R^n$ by Euclidean isometries.
		
		Fix $x\in\R^n$ and a finite symmetric generating set $S$ for $\Gamma$. Proper discontinuity and freeness give a number $\delta>0$ such that the balls
		$$
		\{B(\gamma x,\delta):\gamma\in\Gamma\}
		$$
		are pairwise disjoint. Let
		$
		C:=\max_{s\in S}d(x,sx).
		$
		If the word length of $\gamma$ is at most $k$, then
		$
		d(x,\gamma x)\leq Ck.
		$
		Thus all balls $B(\gamma x,\delta)$ with $|\gamma|_S\leq k$ are contained in $B(x,Ck+\delta)$. Comparing Euclidean volumes gives
		$$
		\#\{\gamma\in\Gamma:|\gamma|_S\leq k\}\,\omega_n\delta^n
		\leq\omega_n(Ck+\delta)^n.
		$$
		Therefore the word-growth function of $\Gamma$ is $O(k^n)$.
	\end{proof}
	
	\subsection{Proof of the main theorems}
	
	\begin{proof}[Proof of Theorem~\ref{thm:contractible-main}]
		Suppose first that $\Ric_g\equiv0$. In dimension three the Riemann curvature tensor is algebraically determined by the Ricci tensor, so $\Rm_g\equiv0$. Since $M$ is contractible, it is simply connected. A complete simply connected flat manifold is isometric to Euclidean space; hence
		$$
		(M,g)\cong(\R^3,g_{\mathrm{eucl}}),
		$$
		and in particular $M$ is diffeomorphic to $\R^3$.
		
		Assume now that $\Ric_g\not\equiv0$. A contractible manifold is one-ended; see \cite[Lemmas 2.1--2.2]{YanZhu}. Theorem~\ref{thm:metric-replacement} therefore gives a complete nonparabolic metric $\widehat g$ on $M$ with $R_{\widehat g}>0$. Corollary~\ref{cor:nonparabolic-contractible}, applied to $(M,\widehat g)$, yields that
		$
		M\cong_{\mathrm{diff}}\R^3.
		$
	\end{proof}
	
	\begin{proof}[Proof of Theorem~\ref{thm:handlebody-main}]
		Suppose first that $\Ric_g\not\equiv0$. The open handlebody $M_\gamma$ is one-ended, so Theorem~\ref{thm:metric-replacement} gives a complete nonparabolic metric $\widehat g$ with positive scalar curvature. Corollary~\ref{cor:nonparabolic-handlebody} then implies $\gamma\leq1$.
		
		It remains to consider $\Ric_g\equiv0$. As above, $(M_\gamma,g)$ is flat. Since $M_\gamma$ deformation retracts onto a wedge of $\gamma$ circles,
		$
		\pi_1(M_\gamma)\cong F_\gamma,
		$
		the free group of rank $\gamma$. Lemma~\ref{lem:flat-group-growth} shows that this group has polynomial growth. For $\gamma\geq2$, however, $F_\gamma$ has exponential word growth: with respect to its standard free generating set, the sphere of radius $k\geq1$ has
		$
		2\gamma(2\gamma-1)^{k-1}
		$
		elements. Hence $\gamma\geq2$ is impossible, and again $\gamma\leq1$.
	\end{proof}
	
	\subsection*{Acknowledgment}
	We greatly appreciate Otis Chodosh for his interest in our work, as well as for providing valuable references and suggestions. We are also grateful to Jian Wang and Jinmin Wang for inviting us to present this work at AMSS, and for the valuable discussions with them.
	
	\subsection*{Declaration of competing interest}
	The authors declare that they have no known competing financial interests in this paper.

	\subsection*{Data availability}
	No data were used for the research described in the article.
	
	\subsection*{On the use of AI}
	Kimi K3 was used only for grammatical polishing, spelling checks, and
	suggestions concerning the wording of section and lemma titles. It was
	not used to generate mathematical ideas, arguments, proofs, calculations. The authors take full responsibility for the entire content of this manuscript.

\end{document}